\documentclass[english ]{article}

\usepackage{custom_tex}

\title{Pre-integration via active subspaces}
\date{February 2022}
\author{Sifan Liu}
\author{Art B. Owen}
\affil{Department of Statistics, Stanford University}

\newcommand{\wal}{{}_b\text{wal}_k}
\newcommand{\walbk}{{}_b\text{wal}_{\bsk}}

\newcommand{\Wal}[1]{{}_b\text{wal}_{#1}}

\renewcommand{\emptyset}{\varnothing}
\renewcommand{\le}{\leqslant}  
\renewcommand{\ge}{\geqslant}
\renewcommand{\leq}{\leqslant}
\renewcommand{\geq}{\geqslant}
\renewcommand{\hk}{\mathrm{HK}}
\renewcommand{\tran}{^{\mathsf{T}}}

\newcommand{\giv}{\!\mid\!}

\newcommand{\bsa}{\boldsymbol{a}}
\newcommand{\bsc}{\boldsymbol{c}}
\newcommand{\bsk}{\boldsymbol{k}}
\newcommand{\bsw}{\boldsymbol{w}}
\newcommand{\bsx}{\boldsymbol{x}}
\newcommand{\bsy}{\boldsymbol{y}}
\newcommand{\bsz}{\boldsymbol{z}}

\newcommand{\dunif}{\mathbb{U}}
\newcommand{\rd}{\,\mathrm{d}}
\newcommand{\dnorm}{\mathcal{N}}

\newcommand{\natu}{\mathbb{N}}
\newcommand{\bvhk}{\mathrm{BVHK}}
\newcommand{\bszero}{\boldsymbol{0}}
\newcommand{\bsl}{\boldsymbol{\ell}}
\newcommand{\olt}{\overline{\tau}}
\newcommand{\ult}{\underline{\tau}}
\newcommand{\wt}{\widetilde}

\newcommand{\phz}{\phantom{0}}

\usepackage{accents}
\DeclareMathSymbol{\widehatsym}{\mathord}{largesymbols}{"62}

\begin{document}
\maketitle

\begin{abstract}
Pre-integration is an extension of conditional Monte Carlo to quasi-Monte Carlo
and randomized quasi-Monte Carlo. It can reduce but not increase the variance
in Monte Carlo.  For quasi-Monte Carlo it can bring about improved regularity
of the integrand with potentially greatly improved accuracy.
Pre-integration is ordinarily done by integrating out one of $d$ input variables
to a function. In the common case of a Gaussian integral one can also
pre-integrate over any linear combination of variables.
We propose to do that and we choose the first eigenvector in an active subspace
decomposition to be the pre-integrated linear combination.
We find in numerical examples that this active subspace pre-integration
strategy is competitive with pre-integrating the first
variable in the principal components construction
on the Asian option where principal components are known to
be very effective.
It outperforms other pre-integration methods on some basket options where there is no
well established default. We show theoretically that, just as in Monte Carlo,
pre-integration can reduce but not increase the variance when one uses
scrambled net integration.  We show that the lead eigenvector in an
active subspace decomposition is closely related to the vector that
maximizes a less computationally tractable criterion using a Sobol' index
to find the most important linear combination of Gaussian variables.
They optimize similar expectations involving the gradient.
We show that the Sobol' index criterion for the leading eigenvector
is invariant to the way that one chooses the remaining $d-1$ eigenvectors
with which to sample the Gaussian vector.
\end{abstract}



\section{Introduction}
\label{sec: intro}

Pre-integration \citep{griewank2018high} is a strategy for high dimensional numerical integration in which one variable is integrated out in closed form
(or by a very accurate quadrature rule)
while the others are handled by quasi-Monte Carlo (QMC) sampling.
This strategy has a long history in Monte Carlo (MC)
sampling \citep{hamm:1956,trot:tuke:1956},
where it is known as conditional Monte Carlo.
In the Markov chain Monte Carlo (MCMC) literature, such
conditioning is called Rao-Blackwellization,
although it does not generally bring the optimal variance
reduction that results from the Rao-Blackwell theorem.
See \cite{robe:robe:2021} for a survey.

The advantage of pre-integration in QMC goes beyond the variance
reduction that arises in MC.  After pre-integration, a $d$-dimensional
integration problem with a discontinuity
or a kink (discontinuity in the gradient)
can be converted into a much smoother $d-1$-dimensional problem
\cite{griewank2018high}.
QMC exploits regularity of the integrand and then smoothness
brings a benefit on top of the $L^2$ norm reduction
that comes from conditioning.
Gilbert, Kuo and Sloan \cite{gilb:kuo:sloa:2021:tr}
point out that the resulting smoothness depends critically on
a monotonicity property of the integrand with respect to
the variable being integrated out.
Hoyt and Owen \cite{hoyt:owen:2020} give conditions where pre-integration
reduces the mean dimension (that we will define below)
of the integrand.
It can reduce mean dimension from proportional to $\sqrt{d}$ to $O(1)$
as $d\to\infty$ in a sequence of ridge functions with a discontinuity that
the pre-integration smooths out.
He \cite{he2019error}  studied the error rate of
pre-integration for scrambled nets applied to functions of the form
$f(\bsx)=h(\bsx)\Indc{\phi(\bsx)\geq 0}$ for a Gaussian variable $\bsx$.
That work assumes that $h$ and $\phi$ are smooth functions and $\phi$ is monotone in $x_{j}$.
Then pre-integration has a smoothing effect that when combined with
some boundary growth conditions yields an error rate of $O(n^{-1+\ep})$.

Many of the use cases for pre-integration involve integration
with respect to the multivariate Gaussian distribution,
especially for problems arising in finance.
In the Gaussian context, we have more choices for the variable
to pre-integrate over.  In addition to pre-integration over any
one of the $d$ coordinates, pre-integration over any
linear combination of the variables remains integration
with respect to a univariate Gaussian variable.
Our proposal is to pre-integrate over
a linear combination of variables chosen to optimize
a measure of variable importance derived from active subspaces
\cite{cons:2015}.

When sampling from a multivariate  Gaussian
distribution by QMC, even without pre-integration,
one must choose a square root of the covariance matrix by which to
multiply some sampled scalar Gaussian variables.
There are numerous choices for that square root.
One can sample via the principal component
matrix decomposition as \cite{acwo:broa:glas:1997} and many others do.
For integrands defined with respect to Brownian motions,
one can use the Brownian bridge
construction studied by~\cite{mosk:cafl:1996}.
These options have some potential disadvantages.
It is always possible that the integrand is little affected
by the first principal component.  In a pessimistic scenario,
the integrand could be proportional to a principal component
that is orthogonal to the first one.
This is a well known pitfall in principal components regression \cite{joll:1982}.
In a related phenomenon,
\cite{papa:2002} exhibits an integrand where QMC via
the standard construction is more effective than via the
Brownian bridge.

Not only might a principal component direction perform poorly, the first principal
component is not necessarily well defined. Although the
problem may be initially defined in terms of a specific
Gaussian distribution, by a change of
variable we can rewrite our integral
as an expectation with respect to another Gaussian distribution with a different
covariance matrix that has a different first principal component.
Or, if the problem is posed with a covariance equal to the $d$-dimensional
identity matrix, then every unit vector linear combination
of variables is a first principal component.

Some proposed methods take account
of the specific integrand while formulating a sampling strategy.
These include stratifying in a direction chosen from exponential tilting \cite{glas:hied:shah:1999},
exploiting a linearization of the integrand at $d+1$ special points
starting with the center of the domain \cite{imai:tan:2004},
and a gradient principal component analysis (GPCA) algorithm \cite{xiao:wang:2019}
that we describe in more detail below.


The problem we consider is to compute an approximation
to $\mu = \e(f(\bsx))$ where $\bsx\in\real^d$
has the spherical Gaussian distribution denoted by $\dnorm(0,I)$
and $f$ has a gradient almost everywhere, that is square integrable.
Let $C = \e( \nabla f(\bsx)\nabla f(\bsx)\tran)\in\real^{d\times d}$.
The $r$ dimensional active subspace \citep{cons:2015} is the space
spanned by the $r$ leading eigenvectors of $C$.
For other uses of the matrix $C$ in numerical computation, see
the references in \citep{constantine2014active}.
For $r=1$, let $\theta$ be the leading eigenvector of $C$ normalized to be a unit vector.
Our use is to  pre-integrate $f$ over $\theta\tran\bsx\sim\dnorm(0,1)$.

The eigendecomposition of $C$ is an uncentered principal components
decomposition of the gradients, also known as the GPCA.
The GPCA method \cite{xiao:wang:2019} also uses the eigendecomposition of $C$
to define a matrix square root for a QMC sampling strategy to reduce
effective dimension, but it involves no pre-integration.
The algorithm in \cite{xiao2018conditional} pre-integrates
the first variable $x_1$ out of $f$.  Then it applies GPCA
to the remaining $d-1$ variables in order to find a suitable
$(d-1)\times(d-1)$ matrix square root for the remaining Gaussian variables.
They pre-integrate over a coordinate variable
while we always pre-integrate
over the leading eigenvector which is not generally one
of the $d$ coordinates.  All of the algorithms that involve $C$
take a sample in order to estimate it.

This paper is organized as follows.
Section~\ref{sec:background} provides some background on RQMC and pre-integration.
Section~\ref{sec:preintvar} shows that pre-integration never increases the variance of
scrambled net integration, extending the well known property of conditional integration
to RQMC.
This holds even if the points being scrambled are not a digital net.
There is presently very little guidance in the literature about which variable to pre-integrate
over, beyond the monotonicity condition recently studied in \cite{gilb:kuo:sloa:2021:tr}
and a remark for ridge functions in \cite{hoyt:owen:2020}.
An RQMC variance formula from Dick and Pillichshammer \cite{dick:pill:2010}
overlaps greatly with a certain Sobol' index and so this
section advocates using the variable that maximizes that Sobol' index of variable importance.
That index has an effective practical estimator due to  Jansen \cite{jans:1999}.
In Section~\ref{sec: active subspace}, we describe using the unit vector $\theta$
which maximizes the Sobol' index for $\theta\tran\bsx$.  This strategy
is to find a matrix square root for which the first column maximizes
the criterion from Section~\ref{sec:preintvar}.
We show that this Sobol' index $\ult^2_\theta$ is well defined in that it does not depend
on how we parameterize the space orthogonal to $\theta$.  This index is upper bounded
by $\theta\tran\e(\nabla f(\bsx)\nabla f(\bsx)\tran)\theta$. When $f$ is differentiable,
the Sobol' index can be expressed as a weighted expectation of $\theta\tran\nabla f(\bsx)\nabla f(\bsx)\tran\theta$.
As a result, choosing a projection by active subspaces amounts to optimizing a computationally
convenient proxy measure for a Sobol' index of variable importance.
We apply active subspace pre-integration to option pricing examples in Section~\ref{sec: option pricing}.
These include an Asian call option and some of its partial derivatives (called the Greeks) as well as a basket option.
The following summary is based on the figures in that section.
Using the standard construction, the active subspace pre-integration is a clear winner for five of the six
Asian option integration problems and is essentially tied with the method from \citep{xiao2018conditional}
for the one called Gamma.  The principal components construction of
\cite{acwo:broa:glas:1997} is more commonly used for this problem
and has been extensively studied.  In that construction active subspace pre-integration is more accurate
than the other methods for Gamma, is worse than some others for Rho and is essentially the
same as the best methods in the other four problems.  Every one of the pre-integration
methods was more accurate than RQMC without pre-integration, consistent
with Theorem~\ref{thm:noharm}.
We also looked at a basket option for which there is not a strong default construction
as well accepted as  the PCA construction is for the Asian option.  There in six cases,
two basket options and three baseline sampling constructions, active subspace pre-integration was always the most accurate.
That section makes some brief comments about the computational cost.
In our simulations, the pre-integration methods cost significantly more than the
other methods, but the cost factor is not enough to overwhelm
their improved accuracy, except that pre-integrating the first
variable in the standard construction of Brownian motion
usually brought little to no advantage.
Section~\ref{sec:discussion} has our conclusions.

%

\section{Background}\label{sec:background}

In this section, we introduce some background of RQMC and pre-integration
and active subspaces.  First we introduce some notations.  Additional
notations are introduced as needed.

For a positive integer $d$, we let $1{:}d=\{1,2,\ldots,d\}$. For a subset $u\subseteq1{:}d$,
let $-u=1{:}d\setminus u$. For an integer $j\in1{:}d$, we use
$j$ to represent $\{j\}$ when the context is clear and $-j=1{:}d\setminus \{j\}$. Let $|u|$ denote the cardinality of $u$.
For
$\bsx\in\R^d$, we let $\bsx_u$ be the $|u|$ dimensional vector containing only the $x_j$
with $j\in u$. For $\bsx,\bsz\in\R^d$ and $u\subseteq1{:}d$, we let $\bsx_u{:}\bsz_{-u}$ be the $d$ dimensional
 vector, whose $j$-th entry is $x_j$ if $j\in u$, and $z_{j}$ if $j\notin u$.
We use $\natu=\{1,2,\ldots\}$ for the natural numbers
and $\natu_0=\natu\cup\{0\}$.
We denote the density and the cumulative distribution function (CDF) of standard Gaussian distribution $\dnorm(0,1)$ as $\varphi$ and $\Phi$, respectively. We let $\Phi^{-1}$ denote the inverse CDF of $\dnorm(0,1)$. We also use $\varphi$ to denote the density of the $d$-dimensional standard Gaussian distribution $\dnorm(0,I_d)$,
$\varphi(\bsx) = (2\pi)^{-d/2}\exp(-\Vert\bsx\Vert^2/2)$.
We use $\dnorm(0,I)$ when the dimension of the random variable is clear
from context.
For a matrix $\Theta\in\real^{d\times d}$ we often need to select
a subset of columns.  We do that via $\Theta[u]$, such as
$\Theta[1{:}r]$ or $\Theta[2{:}d]$.

\subsection{QMC and RQMC}

For background on QMC see the monograph \cite{dick:pill:2010} and
the survey article \cite{dick:kuo:sloa:2013}.  For a survey of RQMC
see \cite{lecu:lemi:2002}.  We will emphasize
scrambled net integration with a recent description
in \cite{sllnrqmc}.  Here we give a brief account.

QMC provides a way to estimate $\mu = \int_{[0,1]^d}f(\bsx)\rd \bsx$
with greater accuracy than can be done by MC while sharing with
MC the ability to handle larger dimensions $d$ than can be well
handled by classical quadrature methods such as those in \cite{davrab}.
The QMC estimate, like the MC one takes the form
$\hat\mu_n = (1/n)\sum_{i=0}^{n-1} f(\bsx_i)$,
except that instead of $\bsx_i\simiid\dunif[0,1]^d$ the sample
points are chosen strategically to get a small value for
$$D_n^*=D_n^*(\bsx_0,\dots,\bsx_{n-1})
=\sup_{\bsa\in[0,1]^d}\Bigl| \frac1n\sum_{i=0}^{n-1}1\{\bsx_i\in[\bszero,\bsa)\}-\prod_{j=1}^da_j\Bigr|$$
which is known as the star discrepancy. The Koksma-Hlawka inequality (see \cite{hick:2014}) is
\begin{align}\label{eq:khi}
|\hat\mu_n-\mu| \le D_n^*\times\Vert f\Vert_{\hk}
\end{align}
where $\Vert \cdot\Vert_{\hk}$ denotes total variation in the sense of Hardy and Krause.
It is possible to construct points $\bsx_i$ with
$D_n^* = O( (\log n)^{d-1}/n)$ or to choose an infinite sequence of them
along which  $D_n^*=O((\log n)^d/n)$.  Both of these are often written as
$O(n^{-1+\epsilon})$ for any $\epsilon>0$ and this rate translates directly
to $|\hat\mu_n-\mu|$ when $f\in\bvhk[0,1]^d$, the set of functions
of bounded variation in the sense of Hardy and Krause.  While the logarithmic
powers are not negligible they correspond to worst case integrands that
are not representative of integrands evaluated in practice, and as we see
below, some RQMC methods provide a measure of control against them.

The QMC methods we study are digital nets and sequences.  To define them,
for an integer base $b\ge$ let
\begin{align}\label{eq:elemint}
E(\bsk,\bsc) = \prod_{j=1}^d\Bigl[ \frac{c_j}{b^{k_j}}, \frac{c_j+1}{b^{k_j}}\Bigr)
\end{align}
for $\bsk = (k_1,\dots,k_d)$ and $\bsc = (c_1,\dots,c_d)$
where $k_j\in\natu_0$ and $0\le a_j <b^{k_j}$. The sets $E(\bsk,\bsc)$
are called elementary intervals in base $b$.
They have volume $b^{-|\bsk|}$ where  $|\bsk|=\sum_{j=1}^dk_j$.
For integers $m\ge t\ge 0$,
the points $\bsx_0,\dots,\bsx_{n-1}$ with $n=b^m$ are a $(t,m,d)$-net
in base $b$ if every elementary interval $E(\bsk,\bsc)$ with $|\bsk|\le m-t$
contains $b^{m-|\bsk|}$ of those points, which is exactly $n$ times
the volume of $E(\bsk,\bsc)$.

For each $\bsk$ the sets $E(\bsk,\bsc)$ partition $[0,1)^d$ into $b^{|\bsk|}$ congruent
half open sets. If $|\bsk|\le m-t$, then the $n$ points $\bsx_i$ are perfectly
stratified over that partition.  The power of digital nets is that the
points $\bsx_i$ satisfy ${m-t+d-1\choose d-1}$ such stratifications simultaneously.
They attain $D_n^* = O((\log n)^{d-1}/n)$ after approximating
each $[\bszero,\bsa)$ by sets $E(\bsk,\bsc)$ \cite{nied:1992}.
Given $b$ and $m$ and $d$, a smaller $t$ is the better. It is not always
possible to get $t=0$.

For an integer $t\ge0$, the infinite sequence $(\bsx_i)_{i\ge0}$ is
a $(t,s)$-sequence in base $b$ if for all  integers $m\ge t$
and $r\ge0$, the points $\bsx_{rb^m},\dots,\bsx_{rb^m+b^m-1}$
are a $(t,m,d)$-net in base $b$.
This means that the first $b^m$ points are a $(t,m,d)$-net
as are the second $b^m$ points and if we take the first $b$
such nets together we get a $(t,m+1,d)$-net.  Similarly
the first $b$ such $(t,m+1,d)$-nets form a $(t,m+2,d)$-net
and so on ad infinitum.

The nets and sequences in common use are called
digital nets and sequences owing to a digital strategy
used in their construction, where $x_{ij}$ is expanded
in base $b$ and there are rules for constructing those
base $b$ digits from the base $b$ representation of $i$.
See \cite{nied:1992}.
The most commonly used nets and sequences are those of Sobol'
\cite{sobo:1967:tran}.
Sobol' sequences are $(t,d)$-sequences in base $2$ and sometimes
using base $2$ brings computational advantages for computers
that work in base $2$.  The value
of $t$ is nondecreasing in $d$.  The first $2^m$ points of a $(t,d)$-sequence
can be a $(t',m,d)$-net for some $t'<t$.

While the Koksma-Hlawka inequality~\eqref{eq:khi} shows that
QMC is asymptotically better than MC for $f\in\bvhk[0,1]^d$ it is not
usable for error estimation, and furthermore, many integrands of
interest such as unbounded ones have $V_{\hk}(f)=\infty$.
RQMC methods can address both problems.  In RQMC the points
$\bsx_i\sim\dunif[0,1]^d$ individually while
collectively they have a small $D_n^*$. The uniformity property makes
\begin{align}\label{eq:muhatrmqc}
\hat\mu_n = \frac1n\sum_{i=0}^{n-1} f(\bsx_i)
\end{align}
an unbiased estimate of $\mu$ when $f\in L^1[0,1]^d$.
If $f\in L^2[0,1]^d$ then $\var(\hat\mu_n)<\infty$ and it can be estimated
by using independent repeated RQMC evaluations.

Scrambled $(t,m,d)$-nets
have $\bsx_i\sim\dunif[0,1]^d$ individually and the collective condition is that they
form a $(t,m,d)$-net with probability one.
For an infinite sequence of $(t,m,d)$-nets one can use
scrambled $(t,d)$-sequences. See \cite{rtms} for both of these.
The estimate $\hat\mu_n$ taken over a scrambled $(t,d)$-sequence satisfies a strong law of large numbers if $f\in L^2[0,1]^{1+\epsilon}$ \cite{sllnrqmc}.
If $f\in L^2[0,1]^d$, then $\var(\hat\mu_n) =o(1/n)$ giving the method
asymptotically unbounded efficiency versus MC which has variance $\sigma^2/n$
for $\sigma^2=\var(f(\bsx))$.
For smooth enough $f$, $\var(\hat\mu_n) = O( n^{-3}(\log n)^{d-1})$ \cite{smoovar,localanti}
with the sharpest sufficient condition in \cite{yue:mao:1999}.
Also there exists
$\Gamma<\infty$ with $\var(\hat\mu_n)\le \Gamma \sigma^2/n$
for all $f\in L^2[0,1]^d$ \cite{snxs}. This bound involves no powers of $\log(n)$.

\subsection{Pre-integration with respect to $\dunif[0,1]^d$ and $\dnorm(0,I)$}

For $j\in1{:}d$,
$\int_{[0,1]^d}f(\bsx)\rd\bsx = \int_{[0,1]^{d-1}} \int_0^1 f(\bsx)\rd x_j\rd\bsx_{-j}$
which we can also write as
$\e(f(\bsx))=\e( \e( f(\bsx)\giv \bsx_{-j}))$ for $\bsx\sim\dunif[0,1]^d$.
For $\bsx\in[0,1]^d$, define
$$
g(\bsx) = g_j(\bsx) = \int_0^1 f(\bsx) \rd x_j.
$$
It simplifies the presentation of some of our results,
especially Theorem~\ref{thm:noharm} on variance reduction,
to keep $g$ defined as above on $[0,1]^d$
even though it does not depend at all on $x_j$ and could be written
as a function of $\bsx_{-j}\in[0,1]^{d-1}$ instead.
In pre-integration
$$
\hat\mu_n =\hat\mu_{n,j} = \frac1n \sum_{i=0}^{n-1} g_j(\bsx_i)
$$
which as we noted in the introduction is conditional MC except that
we now use RQMC inputs.

Pre-integration can bring some advantages for RQMC.
The plain MC variance of $g(\bsx)$ is no larger than that of $f(\bsx)$ and is generally
smaller unless $f$ does not depend at all on  $x_j$. Thus the bound $\Gamma\sigma^2/n$
is reduced.
Next, the pre-integrated integrand $g$ can be much smoother than $f$
and (R)QMC improves on MC by exploiting smoothness.
For example
\cite{griewank2018high,xiao2018conditional} observe that for some option pricing integrands, pre-integrating
certain variable can remove the discontinuities in the integrand or its gradient.

The integrands we consider here are defined with respect to a Gaussian random variable.
We are interested in $\mu=\e( f(\bsz))$ for $\bsz\sim\dnorm(\bszero,\Sigma)$
and a nonsingular covariance $\Sigma\in\real^{d\times d}$.
Letting $R_0\in\real^{d\times d}$ with $R_0R_0\tran=\Sigma$ we can write
$\mu = \e( f(R_0\bsz))$ for $\bsz\sim\dnorm(0,I)$.
For an orthogonal matrix $Q\in\real^{d\times d}$ we also
have $Q\bsz\sim\dnorm(0,I)$.
Then taking $\bsz = \Phi^{-1}(\bsx)$ componentwise leads us to the estimate
$$
\hat\mu =\frac1n\sum_{i=0}^{n-1} f( R\Phi^{-1}(\bsx_i))\quad \text{with $R=R_0Q$}
$$
for RQMC points $\bsx_i$.
The choice of $Q$ or equivalently $R$
does not affect the MC variance of $\hat\mu$
but it can change the RQMC variance. We will consider some examples later.
The mapping $\Phi^{-1}$ from $\dunif[0,1]^d$ to $\dnorm(0,I)$ can be
replaced by another one such as the Box-Muller transformation.
The choice of transformation does not affect the MC
variance but does affect the RQMC variance. Most researchers use
$\Phi^{-1}$ but \cite{okte:gonc:2011} advocates for Box-Muller.

When we are using pre-integration for a problem defined with
respect to a $\dnorm(0,\Sigma)$ random variable we must
choose $R$ and then the coordinate $j$ over which
to pre-integrate.  Our approach is to choose $R$ to make coordinate $j=1$
as important as we can, using active subspaces.

\subsection{The ANOVA decomposition}

For $f\in L^2[0,1]^d$ we can define an analysis of variance (ANOVA) decomposition
from \cite{hoef:1948,sobo:1969,efro:stei:1981}. For details see
\cite[Appendix A.6]{mcbook}.
This decomposition writes
$$
f(\bsx) = \sum_{u\subseteq 1:d}f_u(\bsx)
$$
where $f_u$ depends on $\bsx$ only through $x_j$ with $j\in u$ and also
$\int_0^1f_u(\bsx)\rd x_j =0$ whenever $j\in u$.  The decomposition is orthogonal
in that $\int_{[0,1]^d}f_u(\bsx)f_v(\bsx)\rd\bsx=0$ if $u\ne v$.  The term $f_\emptyset$ is
the constant function everywhere equal to $\mu =\int_{[0,1]^d}f(\bsx)\rd\bsx$.
To each effect $f_u$ there corresponds a variance component
$$
\sigma^2_u = \var(f_u(\bsx)) =
\begin{cases}
\int_{[0,1]^d}f_u(\bsx)^2\rd\bsx, & |u|>0\\
0, & \text{else.}
\end{cases}
$$
For $|u|\ge2$, the effect $f_u$ is called a $|u|$-fold interaction.
The variance components sum to $\sigma^2=\var(f(\bsx))$.
We will use the ANOVA decomposition below when
describing how to choose a pre-integration variable.

Sobol' indices \cite{sobo:1993} are derived from the ANOVA decomposition.
For $u\subseteq1{:}d$ these are
$$
\ult^2_u = \sum_{v\subseteq u}\sigma^2_v\quad\text{and}\quad
\olt^2_u = \sum_{v\subseteq 1:d}\indc{u\cap v\ne\emptyset}\sigma^2_v.
$$
They provide two ways to judge the importance of the set of variables $x_j$ for $j\in u$.
They're usually normalized by $\sigma^2$ to get an interpretation as a proportion of variance explained.

The mean dimension of $f$ is $\nu(f) = \sum_{u\subseteq1:d}|u|\sigma^2_u/\sigma^2$.
It satisfies $\nu(f) = \sum_{j=1}^d\olt^2_j/\sigma^2$.
A ridge function takes the form $f(\bsx) = h(\Theta\tran\bsx)$ for $\Theta\in\real^{d\times r}$
with $\Theta\tran\Theta=I$.  For $\bsx\sim\dnorm(0,I)$, the variance of $f$ does not depend on $d$ and
the mean dimension is $O(1)$ as $d\to\infty$ \cite{hoyt:owen:2020} if $h$ is Lipschitz.
If $h$ has a step discontinuity then it is possible to have $\nu(f)=\Omega(\sqrt{d})$
reduced to $O(1)$ by pre-integration over a component variable $x_j$
with $\theta_j$ bounded away from zero as $d\to\infty$ \cite{hoyt:owen:2020}.

\section{Pre-integration and scrambled net variance}\label{sec:preintvar}

Conditional MC can reduce but not increase the variance of plain MC integration.
Here we show that the same thing holds for scrambled nets using
the nested uniform scrambling of \cite{rtms}. The affine linear scrambling
of \cite{mato:1998:2} has the same variance and hence the same result.
We assume that $f\in L^2[0,1)^d$. The half-open interval is just a notational
convenience.  For any $f\in L^2[0,1]^d$ we could set $f(\bsx)=0$
for any $\bsx\in D=[0,1]^d\setminus[0,1)^d$ and get an equivalent function with
the same integral and, almost surely, the same RQMC estimate because
all $\bsx_i\sim\dunif[0,1]^d$ avoid $D$ with probability one.

We will pre-integrate over one of the $d$ components of $\bsx\in[0,1)^d$.
It is also possible to pre-integrate over
multiple components and reduce the RQMC variance each time,
though the utility of that strategy is limited by the availability
of suitable closed forms or effective quadratures.

\subsection{Walsh function expansions}
To get variance formulas for scrambled nets we follow Dick and Pillichshammer \cite{dick:pill:2010}
who work with a Walsh function expansion of $L^2[0,1)^d$
for which they credit Pirsic \cite{pirs:1995}.
Let $\omega_b=e^{2\pi i/b}$ with $i$ being the imaginary unit.
For $k\in\natu_0$ write $k=\sum_{j\ge0}\kappa_jb^j$
for base $b$ digits $\kappa_j\in\{0,1,\dots,b-1\}$.
For $x\in[0,1)$ write $x=\sum_{k\ge1}\xi_jb^{-j}$ for base $b$
digits $\xi_j\in\{0,1,\dots,b-1\}$.  Countably many $x$ have
an expansion terminating in either infinitely many $0$s or in infinitely many $b-1$s.
For those we always choose the expansion terminating in $0$s.

Using the above notation we can define
the $k$-th $b$-adic Walsh function $\wal:[0,1)\to \bbC$ as
\[
\wal(x)=\omega_b^{\sum_{j\geq 1}\xi_j \kappa_{j-1}}.
\]
The summation in the exponent is  finite because $k<\infty$.
Note that $_b\mathrm{wal}_0(x)=1$ for all $x\in[0,1)$.
For $\bsx=(x_1,\dots,x_d)\in[0,1)^d$ and $\bsk=(k_1,\ldots,k_d)\in\natu_0^d$, the $d$-dimensional Walsh functions
are defined as
\begin{align*}
\walbk(\bsx)=\prod_{j=1}^d {}\Wal{k_j}(x_j).
\end{align*}
The Walsh series expansion of $f(x)$ is
\begin{align*}
f(\bsx)\sim
\sum_{\bsk \in \natu_0^d } \hat f(\bsk)\Wal{\bsk}(\bsx),\  \text{where }
\hat f(\bsk)=\int_{[0,1)^d} f(\bsx)\overline{\Wal{\bsk}(\bsx)} \rd\bsx.
\end{align*}
The $d$-dimensional $b$-adic Walsh function system is a complete orthonormal basis in $L_2([0,1)^d)$ \citep[Theorem A.11]{dick:pill:2010} and the series expansion converges to $f$ in $L^2$.

While our integrand is real valued, it will also satisfy an expansion written
in terms of complex numbers.
For real valued $f$,
\begin{align*}
\var(f)=\sum_{\bsk\in\bbN_0^d\setminus\{\bszero\} } |\hat f(\bsk)|^2.
\end{align*}
The variance under scrambled nets is different.  To study it we group the Walsh coefficients.
For $\bsl\in\natu_0^d$ let
$$
C_{\bsl} = \bigl\{ \bsk\in\natu_0^d\bigm| \lfloor b^{k_j-1}\rfloor \le k_j<b^{\ell_j}, 1\le j\le d\bigr\}.
$$
Then define
$$
\beta_{\bsl}(\bsx) = \sum_{\bsk\in C_{\bsl}}\hat f(\bsk)\walbk(\bsx).
$$
The functions $\beta_{\bsl}$ are orthogonal in that
$\int_{[0,1)^d}\beta_{\bsl}(\bsx)\widebar{\beta_{\bsl'}(\bsx)}\rd\bsx=0$
when $\bsl'\ne\bsl$.
For $\bsl\ne\bszero$, $\beta_{\bsl}(\bsx)$ has variance
$$
\sigma^2_{\bsl} = \int_{[0,1)^d}|\beta_{\bsl}(\bsx)|^2\rd \bsx =
\sum_{\bsk\in C_{\bsl}} |\hat f(\bsk)|^2.
$$
If $\bsx_i$ are a scrambled version of original points $\bsa_i\in[0,1)^d$ then
under this scrambling
\begin{align}\label{eq:generalgain}
\var(\hat\mu_n) = \frac1n\sum_{\bsl\in\natu_0^d\setminus\{\bszero\}} \Gamma_{\bsl}\sigma^2_{\bsl}
\end{align}
for a collection of gain coefficients $\Gamma_{\bsl}\ge0$ that depend on the $\bsa_i$.
This expression can also be obtained through a base $b$ Haar wavelet
decomposition \cite{snetvar}.
Our $\Gamma_{\bsl}$ equals $nG_{\bsl}$ from \cite{dick:pill:2010}.
The variance of $\hat\mu$  under IID MC sampling is $(1/n)\sum_{\bsl\in\natu^d_0\setminus\{\bszero\}}\sigma^2_{\bsl}$
so $\Gamma_{\bsl}<1$ corresponds to integrating the term $\beta_{\bsl}(\bsx)$
with less variance than MC does.

If scrambling of \cite{rtms} or \cite{mato:1998:2} is applied to $\bsa_i$ then
\begin{align}
\label{equ: gain coef}
\Gamma_{\bsl}&=\frac{1}{n}\sum_{i,i'=0}^{n-1}\prod_{j=1}^d \frac{b\Indc{\lfloor b^{\ell_j}a_{i,j}\rfloor = \lfloor b^{\ell_j}a_{i',j}\rfloor } - \Indc{\lfloor b^{\ell_i-1}a_{i,j}\rfloor = \lfloor b^{\ell_j-1}a_{i',j}\rfloor }}{b-1}.
\end{align}
This holds for any $\bsa_i$ not just digital nets.
When $\bsa_i$  are the first $b^m$ points of a $(t,d)$-sequence in base $b$
then $\Gamma=\sup_{\bsl}\Gamma_{\bsl}<\infty$ (uniformly in $m$) so that
$\var(\hat\mu)\le \Gamma\sigma^2/n$.
Similarly for any $\bsl\in\natu_0^d$ we have $\Gamma_{\bsl}\to0$
as $n=b^m\to\infty$ in a $(t,d)$-sequence in base $b$ from which $\var(\hat\mu_n)=o(1/n)$.
For a $(t,m,s)$-net in base $b$, one can show that the gain coefficients $\Gamma_{\bsl}=0$ for all $\bsl$
with $|\bsl|\leq m-t$.

\subsection{Walsh decomposition after pre-integration}

\begin{proposition}
For $f\in L^2[0,1)^d$ and $j\in1{:}d$, let $g$ be $f$ pre-integrated over $x_j$.
Then for $\bsk\in\natu_0^d$,
\begin{align}\label{eq:0ornot}
\hat g(\bsk)=
\begin{cases}
\hat f(\bsk), & k_j=0\\
0,&k_j>0.
\end{cases}
\end{align}
\end{proposition}
\begin{proof}
We write
$$
\hat g(\bsk)
=\int_{[0,1)^{d-1}} \int_0^1g(\bsx)\overline{\walbk(\bsx)}\rd x_j\rd\bsx_{-j}
=\int_{[0,1)^{d-1}} g(\bsx)\int_0^1\overline{\walbk(\bsx)}\rd x_j\rd\bsx_{-j}
$$
because $g(\bsx)$ does not depend on $x_j$.
If $k_j>0$, then the inner integral vanishes establishing the second clause in~\eqref{eq:0ornot}.
If $k_j=0$ then ${}_b\mathrm{wal}_{k_j}(x_j)=1$ for all $x_j$
and the inner integral equals $\prod_{\ell\ne j}{}_b\mathrm{wal}_{k_j}(x_j)=\walbk(\bsx)$
establishing the second clause.
\end{proof}

\begin{theorem}\label{thm:noharm}
For $\bsa_0,\dots,\bsa_{n-1}\in[0,1)^d$ let $\bsx_0,\dots,\bsx_{n-1}$ be a scrambled version
of them using the algorithm from \cite{rtms} or \cite{mato:1998:2}. Let $f\in L^2[0,1)^d$
and for $j\in 1{:}d$, let $g$ be $f$ pre-integrated over $x_j$. Then
$$
\var\Bigl(\frac1n\sum_{i=0}^{n-1}g(\bsx_i)\Bigr)
\le \var\Bigl(\frac1n\sum_{i=0}^{n-1}f(\bsx_i)\Bigr).
$$
\end{theorem}
\begin{proof}
With either $f$ or $g$ we have the same gain coefficients $\Gamma_{\bsl}$ for $\bsl\in\natu_0^d$.
However
\begin{align*}
\sigma^2_{\bsl}(g)
= \sum_{\bsk\in C_{\bsl}}|\hat g(\bsk)|^2
= \sum_{\bsk\in C_{\bsl},k_j=0}|\hat f(\bsk)|^2
\le \sum_{\bsk\in C_{\bsl}}|\hat f(\bsk)|^2=\sigma^2_{\bsl}(f).
\end{align*}
The result now follows from~\eqref{eq:generalgain}.
\end{proof}

Theorem~\ref{thm:noharm} shows that pre-integration does not increase the
variance under scrambling.  This holds whether or not the underlying points
are a digital net, though of course the main case of interest is for scrambling
of digital nets and sequences.

Pre-integration has another benefit that is not captured by Theorem~\ref{thm:noharm}.
By reducing the input dimension from $d$ to $d-1$ we might be able to find
better points in $[0,1]^{d-1}$ than the $(t,m,d)$-net we would otherwise use in $[0,1]^d$.
Those improved points might have some smaller gain coefficients or
they might be a $(t',m,s-1)$ net in base $b$ with $t'<t$.

For scrambled net sampling, reducing the dimension reduces an upper bound on the variance.
For any function $f\in L^2[0,1)^d$, the variance using a scrambled $(t,m,d)$-net in
base $b$ is at most $b^{t+d}$ times the MC variance.
Reducing the dimension reduces the bound to $b^{t+d-1}$ times the MC variance.
For digital constructions in base $2$, there are sharper bounds on these ratios,
$2^{t+d-1}$ and $2^{t+d-2}$, respectively \cite{pan:owen:2021:tr} and for some
nets described there, even lower bounds apply.

As remarked above, pre-integration over a variable $x_j$ that $f(\bsx)$ uses
will reduce the variance under scrambled net sampling. This
reduction does not require $f$ to be monotone in $x_j$, though such
cases have the potential to bring a greater improvement.
Pre-integration can either increase or decrease the mean dimension because
$$
\nu(f) =
\frac{\sum_{\bsk\in\natu_0^d\setminus\{\bszero\}}|\hat f(\bsk)|^2\times\Vert\bsk\Vert_0}
{\sum_{\bsk\in\natu_0^d\setminus\{\bszero\}}|\hat f(\bsk)|^2}
\quad\text{and}\quad
\nu(g_j) =
\frac{\sum_{\bsk\in\natu_0^d\setminus\{\bszero\}}\indc{k_j>0}|\hat f(\bsk)|^2\times\Vert\bsk\Vert_0}
{\sum_{\bsk\in\natu_0^d\setminus\{\bszero\}}\indc{k_j>0}|\hat f(\bsk)|^2}
$$
and pre-integration could possibly reduce the denominator by a greater proportion than
it reduces the numerator.

\subsection{Choice of $x_j$}
In order to choose $x_j$ to pre-integrate over, we can look at the variance
reduction we get.  Pre-integrating over $x_j$ reduces the scrambled net variance by
\begin{align}\label{eq:reducedj}
\frac1n\sum_{\bsl\in\natu_0^d\setminus\{\bszero\}}\Gamma_{\bsl}\sigma^2_{\bsl}\indc{\ell_j>0}
=\frac1n\sum_{\bsk\in\natu_0^d\setminus\{\bszero\}}\Gamma_{\bsl}|\hat f(\bsk)|^2\indc{k_j>0}.
\end{align}
Evaluating this quantity for each $j\in 1{:}d$  might be more expensive than
getting a good estimate of $\mu$. However we don't
need to find the best $j$.
Any $j$ where $f$ depends on $x_j$ will bring some improvement.
Below we develop a principled and computationally convenient choice
by choosing the $j$ which is most important as measured by
a Sobol' index \cite{sobo:1993} from global sensitivity analysis \cite{raza:etal:2021}.



A convenient proxy replacement for~\eqref{eq:reducedj} is
\begin{align}\label{eq:sobolupper}
\frac1n\sum_{\bsk\in\natu_0^d\setminus\{\bszero\}}
|\hat f(\bsk)|^2\indc{k_j>0}
=\frac1n\sum_{u\subseteq1:d}\sigma^2_u\indc{j\in u}
\end{align}
where $\sigma^2_u$ is the ANOVA variance component for the set $u$.
The equality above follows because the ANOVA can be defined,
as Sobol' \cite{sobo:1969} did, by collecting up the terms involving $x_j$ for $j\in u$
from the orthogonal decomposition.  Sobol' used Haar functions.
The right hand side of~\eqref{eq:sobolupper} equals $\olt^2_j/n$.
It counts all the variance components in which variable $j$ participates.
From the orthogonality properties of ANOVA effects it follows that
$$
\olt^2_j = \frac12\int_{[0,1]^{d+1}} \bigl(f(\bsz_j{:}\bsx_{-j})-f(\bsx)\bigr)^2\rd\bsz_j\rd\bsx.
$$
The Jansen estimator \cite{jans:1999} is an estimate of the above integral that can
be done by a $d+1$ dimensional MC or QMC or RQMC sampling algorithm.
Our main interest in this Sobol' index estimator is that we use it as a point
of comparison to the use of active subspaces in choosing a projection of
a Gaussian vector along which to pre-integrate.

\section{Active subspace method}\label{sec: active subspace}

Without loss of generality an expectation defined with respect to $\bsx\sim\dnorm(0,\Sigma)$ for
nonsingular $\Sigma\in\real^{d\times d}$ can be written as an expectation with respect
to $\bsx\sim\dnorm(0,I)$.
For a unit vector $\theta\in\real^d$, we will pre-integrate over $\bsx\tran\theta\sim\dnorm(0,1)$ and
then the problem is to make a principled choice of $\theta$.
It would not be practical to seek an optimal choice.

Our proposal is to use active subspaces \cite{cons:2015}.  As mentioned in the introduction we let
$$C = \e( \nabla f(\bsx)\nabla f(\bsx)\tran)$$
and then let $\Theta[1{:}r]$ comprise  the $r$ leading eigenvectors of $C$.
The original use for active subspaces is to approximate $f(\bsx)\approx
\tilde f(\Theta[1{:}r]\tran\bsx)$ for some function $\tilde f$ on $\real^r$.
It is well known that one can construct functions where the active subspace will
be a bad choice over which to approximate.
For instance, with $f(\bsx) = \sin(10^6x_1)+100x_2$ the $r=1$ active
subspace provides a function of $x_1$ alone while a function of $x_2$ alone
can provide a better approximation than a function of $x_1$ alone can.
Active subspaces remain useful for approximation because the motivating problems
are not so pathological and there is a human in the loop to catch such things.
They also have an enormous practical advantage that one
set of evaluations of $\nabla f$
can be used in the search for $\Theta$ instead of having every candidate
$\Theta$ require its own evaluations of $\nabla f$.
Using active subspaces for integration retains that advantage.

In our setting, we take $r=1$ and pre-integrate over $\theta\tran\bsx$ where
$\theta$ is the leading eigenvector of $C$.  That is $\theta$ maximizes
$\theta\tran \e(\nabla f(\bsx)\nabla f(\bsx)\tran)\theta$ over $d$ dimensional unit vectors.
Now suppose that instead of using $f(\bsx)$ we use $f_Q(\bsx)=f(Q\bsx)$
for an orthogonal matrix $Q\in\real^{d\times d}$.
Then $\e( \nabla  f_Q(\bsx) \nabla  f_Q(\bsx)\tran ) = Q\tran CQ$
which is similar to $C$.  It has the same eigenvalues and the leading eigenvector
is $\tilde\theta =Q\tran\theta$.
Furthermore, Theorem 3.1 of \cite{zhan:wang:he:2021} shows that
the invariance extends to the whole eigendecomposition.

\subsection{Connection to a Sobol' index}
The discussion in Section~\ref{sec:preintvar} motivates
pre-integration of $f(\bsx)$ for $\bsx\sim\dnorm(0,I)$
over a linear combination $\theta\tran\bsx$
having the largest Sobol' index over unit vectors $\theta$.
For $\theta_1=\theta$,
let $\Theta=(\theta_1,\theta_2,\dots,\theta_d)\in\real^{d\times d}$ be an orthogonal matrix
and write
$$f_\Theta(\bsx)=f(\Theta\bsx)=f(x_1\theta_1+x_2\theta_2+\cdots+x_d\theta_d).$$
Then we define $\olt^2_\theta(f)$ to be $\olt^2_1$ in the ANOVA
of $f_\Theta(\bsx)$.
First we show that $\olt^2_\theta$ does not depend on
the last $d-1$ columns of $\Theta$.

Let $z$, $\tilde z$ and $\bsy$ be independent
with distributions $\dnorm(0,1)$, $\dnorm(0,1)$ and $\dnorm(0,I_{d-1})$
respectively.
Let $\bsx = z\theta_1+\Theta_{-1}\bsy$ and
$\tilde \bsx = \tilde z\theta_1 + \Theta_{-1}\bsy$.
Using the Jansen formula for $\theta=\theta_1$,
\begin{align}\label{eq:beforechangevariable}
\olt^2_\theta = \frac12\int_\real\int_\real\int_{\real^{d-1}}
\bigl( f(z\theta_1 +\Theta_{-1}\bsy)-f(\tilde z\theta_1+\Theta_{-1}\bsy)\bigr)^2
\varphi(\bsy)\rd\bsy\varphi(\tilde z)\rd \tilde z\varphi(z)\rd z.
\end{align}
Now for an orthogonal matrix $Q\in\real^{(d-1)\times(d-1)}$, let
$$\wt \Theta = \Theta\begin{pmatrix}
1 & \bszero_{d-1}\tran\\
\bszero_{d-1} & Q
\end{pmatrix}
=\begin{pmatrix}
\tilde\theta_1&\tilde\theta_2 &\cdots& \tilde \theta_d
\end{pmatrix}
$$
where $\tilde \theta_1=\theta_1$.
In this parameterization we get
\begin{align*}
\olt^2_\theta
&= \frac12\int_\real\int_\real\int_{\real^{d-1}}
\bigl( f(z\theta_1 +\wt\Theta_{-1}\bsy)-f(\tilde z\theta_1+\wt\Theta_{-1}\bsy)\bigr)^2
\varphi(\bsy)\rd\bsy\varphi(\tilde z)\rd \tilde z\varphi(z)\rd z\\
&= \frac12\int_\real\int_\real\int_{\real^{d-1}}
\bigl( f(z\theta_1 +\Theta_{-1}Q\bsy)-f(\tilde z\theta_1+\Theta_{-1}Q\bsy)\bigr)^2
\varphi(\bsy)\rd\bsy\varphi(\tilde z)\rd \tilde z\varphi(z)\rd z
\end{align*}
which matches~\eqref{eq:beforechangevariable}
after a change of variable.
There is an even stronger invariance property in this setup.
The random variable $\e( f_\Theta(\bsx)\giv \bsx_{-1})$
(random because it depends on $\bsx_{-1}$) has a distribution
that does not depend on $\theta_2,\dots,\theta_d$.

\begin{theorem}\label{thm:uniquesobol}
Let $\bsx\sim\dnorm(0,I_d)$ for $f$ with $\e(f(\bsx)^2)<\infty$
and let $\Theta\in\real^{d\times d}$ be an orthogonal matrix
with columns $\theta_j$ for $j=1,\dots,d$.
Then the distribution of $\e( f_\Theta(\bsx)\giv\bsx_{-1})$
does not depend on the last $d-1$ columns of~$\Theta$.
\end{theorem}
\begin{proof}
Let $\Theta = (\theta_1\ \Theta_{-1})$
and $\wt\Theta=(\theta_1\ \wt\Theta_{-1})$
be orthogonal $d\times d$ matrices.
Define
$$f_\Theta(\bsx)=f(\Theta\bsx)=
f(x_1\theta_1+\Theta_{-1}\bsx_{-1})
\quad\text{and}\quad f_{\wt\Theta}(\bsx)=
f(\wt\Theta\bsx)=
f(x_1\theta_1 + \wt\Theta_{-1}\bsx_{-1}).
$$

Now $\Theta_{-1}\bsx_{-1}$ and $\wt\Theta_{-1}\bsx_{-1}$
both have the same $\dnorm(0,I-\theta_1\theta_1\tran)$ distribution
independently of $x_1\theta_1$.
Choose  $A_1\subset \real$ and $A_0$ in the support of
$\dnorm(0,I-\theta_1\theta_1\tran)$ with positive probability under
that (singular) distribution.
Then $\Pr( f_{\Theta}(\bsx)\in A_1\giv \Theta_{-1}\bsx_{-1}\in A_0)
=\Pr(f_{\wt\Theta}(\bsx)\in A_1\giv \wt\Theta_{-1}\bsx_{-1}\in A_0)$.

For $A\subset \real$ and $B\subset \real^{d-1}$,
\begin{align*}
\Pr( f_\Theta (\bsx)\in A\giv \bsx_{-1}\in B)
&=\Pr( f_\Theta(\bsx)\in A\giv \Theta_{-1}\bsx_{-1}\in \Theta_{-1}B)\\
&=\Pr( f_{\wt\Theta}(\bsx)\in A\giv \wt\Theta_{-1}\bsx_{-1}\in \Theta_{-1}B)\\
&=\Pr( f_{\wt\Theta}(\bsx)\in A\giv \Theta_{-1}\tran\wt\Theta_{-1}\bsx_{-1}\in B)
\end{align*}
where $\Theta_{-1}\tran\wt\Theta_{-1}\bsx_{-1}$
has the same $\dnorm(0,I_{d-1})$   distribution that $\bsx_{-1}$ has.
Then for $C\subset[0,1]$,
\begin{align*}
\Pr\bigl( \Pr( f_\Theta(\bsx)\in A\giv \bsx_{-1}\in B)\in C\bigr)
&=\Pr\bigl( \Pr( f_{\wt\Theta}(\bsx)\in A\giv \Theta_{-1}\tran\wt\Theta_{-1}\bsx_{-1}\in B)\in C\bigr)\\
&=\Pr\bigl( \Pr( f_{\wt\Theta}(\bsx)\in A\giv \bsx_{-1}\in B)\in C\bigr).
\end{align*}
It follows that the distribution of $\Pr( f_{\wt\Theta}(\bsx)\in A\giv \bsx_{-1})$
is the same as the distribution of $\Pr( f_\Theta(\bsx)\in A\giv \bsx_{-1})$.
Integrating over $\real$ we get that $\e( f_{\wt\Theta}(\bsx)\in A\giv \bsx_{-1})$
has the same distribution as  $\e(f_\Theta(\bsx)\in A\giv \bsx_{-1})$ does.
\end{proof}

Another consequence of Theorem~\ref{thm:uniquesobol}
is that $\ult^2_\theta(f)=\ult^2_1(f_\Theta)$ is unaffected by $\theta_2,\dots,\theta_d$.
Because the variance of $f$ is  unchanged by making an orthogonal matrix
transformation of its inputs,  the normalized Sobol' indices
$\ult^2_\theta/\sigma^2$ and $\olt^2_\theta/\sigma^2$ are also
invariant.

Finding the optimal $\theta$ would ordinarily require an expensive search
because every estimate of  $\olt_{\theta}^2$, for a given $\theta$ would require its own
collection of evaluations of $f$.
Using a Poincar\'e inequality in \cite{im2009derivative} we can bound that Sobol' index by
\begin{align*}
\olt_{\theta}^2(f)&
\le \e(\theta\tran\nabla f(\bsx))^2)
=\theta\tran C\theta.
\end{align*}
The active subspace direction thus maximizes an upper bound on the Sobol' index
for a projection. Next we develop a deeper correspondence between these two measures.

For a unit vector $\theta\in\real^d$, we can
write $f(\bsx) = f(\theta\theta\tran\bsx + (I-\theta\theta\tran)\bsx)$.
If $\bsx,\bsz$ are independent $\dnorm(0,I)$ vectors then we can change
the component of $\bsx$ parallel to $\theta$ by changing the argument
of $f$ to be
$\theta\theta\tran\bsz + (I-\theta\theta\tran)\bsx$.
This leaves the resulting point unchanged in the $d-1$ dimensional space
orthogonal to $\theta$.
Let $\tilde x=\theta\tran\bsx$ and $\tilde z=\theta\tran\bsz$.
Then $\tilde x, \tilde z\sim\dnorm(0,1)$ and $(I-\theta\theta\tran)\bsx\sim\dnorm(0,I-\theta\theta\tran)$
are all independent.  If $f$ is differentiable, then by the mean value theorem
\begin{align*}
&f(\theta\theta\tran\bsz + (I-\theta\theta\tran)\bsx)
-f(\theta\theta\tran\bsx + (I-\theta\theta\tran)\bsx)
=
\theta\tran\nabla f(\theta\tilde y + (I-\theta\theta\tran)\bsx)(\tilde z-\tilde x)
\end{align*}
for a real number $\tilde y$ between $\tilde x$ and $\tilde z$.
Using the Jansen formula, the Sobol' index for this projection is
\begin{align}\label{eq:sobolforprojection}
\frac12
\theta\tran
\e\Bigl(
(\tilde z-\tilde x)^2
\nabla f(\theta\tilde y + (I-\theta\theta\tran)\bsx)
\nabla f(\theta\tilde y + (I-\theta\theta\tran)\bsx)\tran
\Bigr)
\theta
\end{align}
which matches $\theta\tran\e(\nabla f(\bsx)\nabla f(\bsx)\tran)\theta$
over a $d-1$ dimensional subspace
but differs from it as follows.  First, it includes a weight factor $(\tilde z-\tilde x)^2$
that puts more emphasis on pairs of inputs where
$\theta\tran\bsx$ and $\theta\tran\bsz$ are far from each other.
Second, the evaluation point projected onto $\theta$ equals $\tilde y$
which lies between two independent $\dnorm(0,1)$
variables instead of having the $\dnorm(0,1)$ distribution, and just
where it lies between them depends on details of $f$ and there could
be more than one such $\tilde y$ for some $f$.
The formula simplifies in an illustrative way for quadratic functions $f$.

\begin{proposition}\label{prop:sobolforquadratic}
If $f:\real^d\to\real$ is a quadratic function and $\theta\in\real^d$
is a unit vector,  then the Sobol' index $\olt_\theta^2$ is
\begin{align}\label{eq:sobolforquadratic}
\theta\tran
\e\biggl(
\nabla f\Bigl(\frac{\theta\theta\tran\bsx}{\sqrt{2}} + (I-\theta\theta\tran)\bsx\Bigr)
\nabla f\Bigl(\frac{\theta\theta\tran\bsx}{\sqrt{2}} + (I-\theta\theta\tran)\bsx\Bigr)\tran
\biggr)
\theta.
\end{align}
\end{proposition}
\begin{proof}
If $f$ is quadratic, then $\tilde y = (\tilde z+\tilde x)/2\sim\dnorm(0,1/2)$ and
$\tilde z-\tilde x\sim\dnorm(0,2)$ and $(I-\theta\theta\tran)\bsx$
are all independent.  Then $\e((\tilde z-\tilde x)^2)=2$
and $\theta\tilde y$ has the same distribution as $\theta\theta\tran\bsx/\sqrt{2}$
which is also independent of $(\tilde z-\tilde x)$ and $(I-\theta\theta\tran)\bsx$.
Making those substitutions in~\eqref{eq:sobolforprojection} yields~\eqref{eq:sobolforquadratic}.
\end{proof}

The Sobol' index in equation~\eqref{eq:sobolforquadratic}
matches the quantity optimized by the first active subspace
apart from the divisor $\sqrt{2}$ affecting one of the $d$ dimensions.
We can also show directly that for $\bsx\sim\dnorm(0,I)$
and $f(\bsx)=(1/2)\bsx\tran A\bsx+b\tran \bsx$ for
a symmetric matrix $A$, that the Sobol' criterion
reduces to $\theta\tran A^2\theta +(\theta\tran b)^2-(1/2)(\theta\tran A\theta)^2$
compared to an active subspace criterion of
$\theta\tran A^2\theta +(\theta\tran b)^2$.

\subsection{Active subspace}
Because $C=\e(\nabla f(\bsx)\nabla f(\bsx)\tran)$ is positive semi-definite (PSD), it has the eigen-decomposition $C=\Theta D \Theta\tran$, where $\Theta=(\theta_1,\ldots,\theta_d)\in\real^{d\times d}$ is an orthogonal matrix consisted of eigenvectors of $C$,
and $D=\text{diag}(\lambda_1,\ldots,\lambda_d)$ with $\lambda_1\geq \ldots\geq\lambda_d\geq 0$ being the eigenvalues.
Constantine et al.\,\cite{constantine2014active} prove that there exists a constant $c$ such that
\begin{align}
\e\bigl(\bigl( f(\bsx) - \e(f(\bsx)\giv \Theta[1{:}r]\tran \bsx)\bigr)^2\bigr)
\leq c(\lambda_{r+1}+\cdots+\lambda_{d})
\label{equ: poincare}
\end{align}
for all $f$ with a square integrable gradient.
In general, the Poincar\'e constant $c$ depends on the support of the function and the probability measure. But for multivariate standard Gaussian distribution, the Poincar\'e constant is always 1 \citep{chen1982inequality,parente2020generalized}. This is because $\Theta[1{:}r]\tran \bsx$ and $\Theta[-(1{:}r)]\tran \bsx$ are independent standard Gaussian variables thus
\begin{align*}
\e\bigl(\bigl( f(\bsx) - \e(f(\bsx)\giv \Theta[1{:}r]\tran \bsx)\bigr)^2\mid \Theta[1{:}r]\tran \bsx\bigr )
\leq \lambda_{r+1}+\cdots+\lambda_{d}
\end{align*}
for all $\Theta[1{:}r]\tran\bsx$.


In our problem, we take $r=1$ and use
$\e( f(\bsx)\giv \theta\tran\bsx)$ where $\theta$
is the first column of $\Theta$.
Because we will end up with a $d-1$ dimensional integration problem it is convenient
in an implementation to make $\theta\tran\bsx$ the last variable not the first.
For instance, one would use the first $d-1$ components in a Sobol' sequence
not components $2$ through $d$. 
Taking $\theta$ to be the first column of  $\Theta$, we compute with
$$g(\bsx_{-d})=
\int_{-\infty}^\infty f(\theta x_d+\Psi\bsx_{-d})\rd x_d$$
using a quadrature rule of negligible error or a closed form expression
if a suitable one is available for an orthonormal
matrix $\Psi\in\real^{d\times(d-1)}$ that is orthogonal to $\theta$.
We then integrate this $g$ over $d-1$ variables by RQMC.
We can use $\Psi=\Theta[2{:}d]$.
Or if we want to avoid the cost of computing the
full eigendecomposition of $C$ we can find $\theta_1$
by a power iteration and then use a Householder transformation
\[
\Theta=I-2\bsw\bsw\tran,\quad \text{where}\quad\bsw=\frac{\theta-e_1}{\|\theta-e_1\|}
\]
and $e_1=(1,0,0,\ldots,0)\tran$.
This $\Theta$ is an orthogonal matrix whose first column is $\theta$
and again we can choose $\Psi = \Theta[2{:}d]$.
In our numerical work, we have used $\Theta[2{:}d]$ instead of the Householder transformation
because of the effective dimension
motivation for those eigenvectors given by \cite{xiao:wang:2019}.

In practice, we must estimate $C$.
In the above description, we replace $C$ by
\begin{align}
\wh C =\frac{1}{M}\sum_{i=0}^{M-1} \nabla f(\bsx_i)\nabla f(\bsx_i)\tran,
\label{equ: Sigma hat}
\end{align}
for an RQMC generated sample with $\bsx_i\sim\dnorm(0,I_d)$
and then define $\theta$ and $\Theta_{-1}$ using $\wh C$ in place of $C$.
We summarize the procedure in Algorithm \ref{algo}.
\begin{algorithm}[t]
\caption{pre-integration with active subspace}
\label{algo}
\SetKwInOut{Input}{Input}
\SetKwInOut{Output}{Output}
\Input{Integrand $f$, number of samples $M$ to compute $\hat C$, number of samples $n$ to compute $\hat\mu$}
\Output{An estimate $\hat\mu$ of $\int_{\R^d}f(\bsx)\varphi(\bsx) d\bsx $}
\tcc{Find active subspaces}
Take $\bsx_0,\ldots,\bsx_{M-1}\sim\dnorm(0,I_d)$ by RQMC.\\
Compute $\wh C=\frac{1}{M}\sum_{i=0}^{M-1}\nabla f(\bsx_i)\nabla f(\bsx_i)\tran $.\\
Compute the eigen-decomposition $\wh C=\wh \Theta \wh D\wh \Theta\tran$.\\

\tcc{Pre-integration}
Let $\theta$ be the first column of $\wh \Theta$\\
Compute the pre-integrated function $\e( f(\bsx)\giv \theta\tran\bsx)$ by a closed form or quadrature rule
\\
\tcc{RQMC integration}
Take $\bsx_{0},\ldots,\bsx_{n-1}\sim\dnorm(0,I_{d-1})$ by RQMC.\\
Let $\hat\mu=\frac{1}{n}\sum_{i=0}^{n-1} \tf(\bsx_i)$.
\end{algorithm}



Using our prior notation we can now describe the approach
of \cite{xiao2018conditional} more precisely.
They first pre-integrate one variable in closed form producing a $d-1$ dimensional
integrand.
They then apply gradient GPCA  to the pre-integrated function to find a good $d-1$
dimensional rotation matrix.
\cite{constantine2014active}.
That is, they first find $h(\bsx_{2:d}):=\e(f(\bsx)\giv\bsx_{2:d})$, then compute
\begin{align*}
\widehat{\widetilde C}=\frac{1}{M}\sum_{i=0}^{M-1}\nabla h(\bsx_i)\nabla h(\bsx_i)\tran\in\real^{(d-1)\times (d-1)},\quad
 \bsx_i\sim\dnorm(0,I_{d-1}),\numberthis\label{equ: hat tilde C}
\end{align*}
using RQMC points $\bsx_i$.
Then they find the eigen-decomposition $\widehat{\widetilde C}=\wh V\wh \Lambda \wh V\tran$. Finally, they use RQMC to integrate the function $h(\wh V\bsx)$ where $\bsx\sim\dnorm(0,I_{d-1})$.
The main difference is that they apply pre-integration to the original integrand $f(\bsx)$ while we apply pre-integration to the rotated integrand $f_\Theta(\bsx)=f(\Theta\bsx)$. They conduct GPCA in the end as an approach to reduce effective dimension, while we conduct a similar GPCA (active subspace method) at the beginning to find the important subspace.

\section{Application to option pricing}\label{sec: option pricing}

Here we study some Gaussian integrals arising from financial valuation.
We assume that an asset price $S_t$, such as a stock, follows a geometric Brownian motion satisfying the stochastic differential equation (SDE)
\begin{align*}
\mrd S_t=r S_t\mrd t+\sigma S_t \mrd B_t,
\end{align*}
where $B_t$ is a Brownian motion. Here, $r$ is the interest rate and $\sigma>0$ is the constant volatility for the asset.
For an initial price $S_0$, the SDE above has a unique solution
\begin{align*}
S_t=S_0\exp\Bigl(\Bigl(r-\frac{\sigma^2}{2}\Bigr)t+\sigma B_t\Bigr).
\end{align*}
Suppose the maturity time of the option is $T$.
In practice, we simulate discrete Brownian motion. We call $B$ a $d$-dimensional discrete Brownian motion if $B$ follows a multivariate Gaussian distribution with mean zero and covariance $\Sigma$ with $\Sigma_{ij}=\Delta t\min(i, j)$, where $\Delta t=T/d$ is the length of each time interval
and $1\le i,j\le d$. To sample a discrete Brownian motion, we can first find a $d\times d$ matrix $R$ such that $RR\tran=\Sigma$, then generate a standard Gaussian variable $\bsz\sim\dnorm(0,I_d)$, and let $B=R\bsz$. 
 Taking $R$ to be the lower triangular matrix in the Cholesky decomposition of $\Sigma$ yields the \emph{standard construction}.
Using the usual eigen-decomposition $\Sigma=U\Lambda U\tran$,
we can take $R=U\Lambda^{1/2}$. This is called the \emph{principal component analysis (PCA)} construction.
For explicit forms of both these choices of $R$, see \cite{glasserman2004monte}.

\subsection{Option with one asset}

When we use the matrix $R$, we can approximate $S_{j\Delta t}$ by
\begin{align*}
S_j&=S_0\exp\Bigl(\Bigl(r-\frac{\sigma^2}{2}\Bigr)j\Delta t + \sigma B_{j}\Bigr),
\quad 1\le j\le d
\end{align*}
where $B=R\bsz$ is the discrete Brownian motion. The arithmetic average of the stock price is given by
\begin{align*}
\bar S(R,\bsz)= \frac{S_0}{s}\sum_{j=1}^d\exp\Bigl(\Bigl(r-\frac{\sigma^2}{2}\Bigr)j\Delta t+\sigma\sum_{k=1}^d R_{jk}z_k \Bigr).
\end{align*}
Then the expected payoff of the arithmetic average Asian call option with strike price $K$ is $\e\bigl((\bar S(A,\bsz)-K)_+\bigr)$,
where the expectation is taken over $\bsz\sim\dnorm(0,I_d)$.

Suppose that we want to marginalize over $z_1$ before computing the expectation
$\e((\bar S(A,\bsz)-K)_+)$. If $R_{j1}>0$ for all $1\leq j\leq d$, then $\bar S(R,\bsz)$ is increasing in $z_1$ for any value of $\bsz_{2:s}$. If we can find $\gamma=\gamma(\bsz_{2:s})$ such that
\begin{align}
\bar S(R,(\gamma,\bsz_{2:s}))=K,
\label{equ: S=K}
\end{align}
then the pre-integration step becomes
\begin{align*}
&\e((\bar S(A,\bsz)-K)_+\giv \bsz_{2:s} )\\
&= \int_{z_1\geq \gamma(\bsz_{2:s})} (\bar S(R,(z_1,\bsz_{2:s}))-K)\varphi(z_1)\rd z_1\\
&=
\frac{S_0}{d}\sum_{j=1}^d\exp\Bigl(\Bigl(r-\frac{\sigma^2}{2}\Bigr)j\Delta t + \sigma\sum_{k=2}^dR_{jk}z_k+\frac{\sigma^2R_{j1}^2}{2} \Bigr)\bar\Phi(\gamma-\sigma R_{j1})
-K\bar\Phi(\gamma),
\numberthis\label{equ: pre-int}
\end{align*}
where $\bar\Phi(x)=1-\Phi(x)$.
In practice, Equation \eqref{equ: S=K} can be solved by a root finding algorithm. For example, Newton iteration usually converges in only a few steps.

The condition that $R_{j1}>0$ for $1\le j\le d$ is satisfied when we use the standard construction or PCA construction of Brownian motion. Using the active subspace method, we are using
$\tilde R=R\Theta$ in the place of $R$
where $\Theta$ consists of the eigenvectors of
$C=\e(\nabla f(\bsz)\nabla f(\bsz)\tran)$, and here $f(\bsz)=(\bar S(A,\bsz)-K)_+$.
If every $\tR_{j1}$ is negative then we replace $\tR$ by $-\tR$.
We have not proved that the components of the first column of $\tR$ must
all have the same sign,
but that has always held for the integrands in our simulations.
As a result we have not had to use a numerical quadrature.


We compare Algorithm \ref{algo} with other methods in the option pricing example considered in \cite{xiao2018conditional} and \cite{he2019error}. Apart from the payoff function of call option, we also consider the Greeks: Delta, Gamma, Rho, Theta, and Vega.  These are defined in \cite{xiao2018conditional}.
We take the parameters $d=50$, $T=1$, $\sigma=0.4$, $r=0.1$, $S_0=K=100$ the same as in \cite{he2019error}.
We consider 4 methods:
\begin{itemize}
        \item \texttt{AS+pre-int}: our proposed active subspace pre-integration method (Algorithm \ref{algo}), which applies active subspace method to find the direction to pre-integrate,
        \item \texttt{pre-int+DimRed}: the method proposed in
\cite{xiao2018conditional}, which first pre-integrates $z_1$ and applies GPCA to conduct dimension reduction for the other
$d-1$ variables,
        \item \texttt{pre-int}: pre-integrating $z_1$ with no dimension reduction,
        \item \texttt{RQMC}: usual RQMC, and
        \item \texttt{MC}: plain Monte Carlo.
\end{itemize}

We vary $n$ from $2^3$ to $2^{17}$. For each $n$, we repeat the simulation 50 times and compute the root mean squared error (RMSE) defined by
\[
\sqrt{\frac{1}{50}\sum_{k=1}^{50}(\hat\mu^{(k)}-\mu)^2},
\]
where $\hat\mu^{(k)}$ is the estimate in the $k$-th replicate.
The true value $\mu$ is approximated by applying \texttt{pre-int+DimRed} using PCA construction with $n=2^{17}$ RQMC samples and averaging over 30 independent replicates. We chose this one because it works well and we wanted to
avoid using \texttt{AS-pre-int} in case the model used for ground
truth had some advantage in reported accuracy.

The root mean square error (RMSE) is plotted versus the sample size on the log-log scale.
For the methods \texttt{AS+pre-int} and \texttt{pre-int+DimRed}, we use $M=128$ samples to estimate $C$ as in \eqref{equ: Sigma hat}.
We approximate the gradients of the original integrand and the pre-integrated integrand by the finite difference
\[
\nabla f(x)\approx \left(\frac{f(x+\ep \bfe_1)-f(x)}{\ep},\ldots,\frac{f(x+\ep \bfe_{d-1})-f(x)}{\ep}
\right)\tran,\quad \ep=10^{-6},
\]
matching the choice in \cite{xiao2018conditional}.
We chose a small value of $M$ to keep the costs comparable to plain RQMC. Also
because $\theta$ is a local optimum of $\theta\tran C\theta$,
we have $\hat\theta\tran C\hat\theta=\theta\tran C\theta+
O(\Vert\hat\theta-\theta\Vert^2)$ so there are diminishing
returns to accurate estimation of $\theta$.
Finally, any $\theta$ where $f$ varies along $\theta\tran\bsx$
brings a variance reduction.

We consider both the standard construction (Figure \ref{fig: std}) and the PCA construction (Figure \ref{fig: pca}) of Brownian motion.
Several observations are in order:
\begin{enumerate}[(a)]
\item With the standard construction,  \texttt{AS+pre-int} dominates all the other methods for five of the six test functions and is tied for best with \texttt{pre-int+DimRed}
for the other one (Gamma).

\item With the PCA construction,  \texttt{AS+pre-int},  \texttt{pre-int+DimRed}
and \texttt{pre-int} are the best methods
for the payoff, Delta, Theta, and Vega
and are nearly equally good.

\item For Rho, \texttt{pre-int+DimRed} and \texttt{pre-int}
are best, while for Gamma, \texttt{AS+pre-int} is best.
\end{enumerate}

The performance of active subspace pre-integration
is the same under either the standard or the PCA
construction by invariance.
For these Asian options it is already well known
that the PCA is especially effective.
Active subspace pre-integration finds something
almost as good without special knowledge, coming out
better in one example, worse in another and essentially
the same in the other four.

\begin{figure}
\begin{subfigure}{.5\textwidth}
\includegraphics[width=\textwidth]{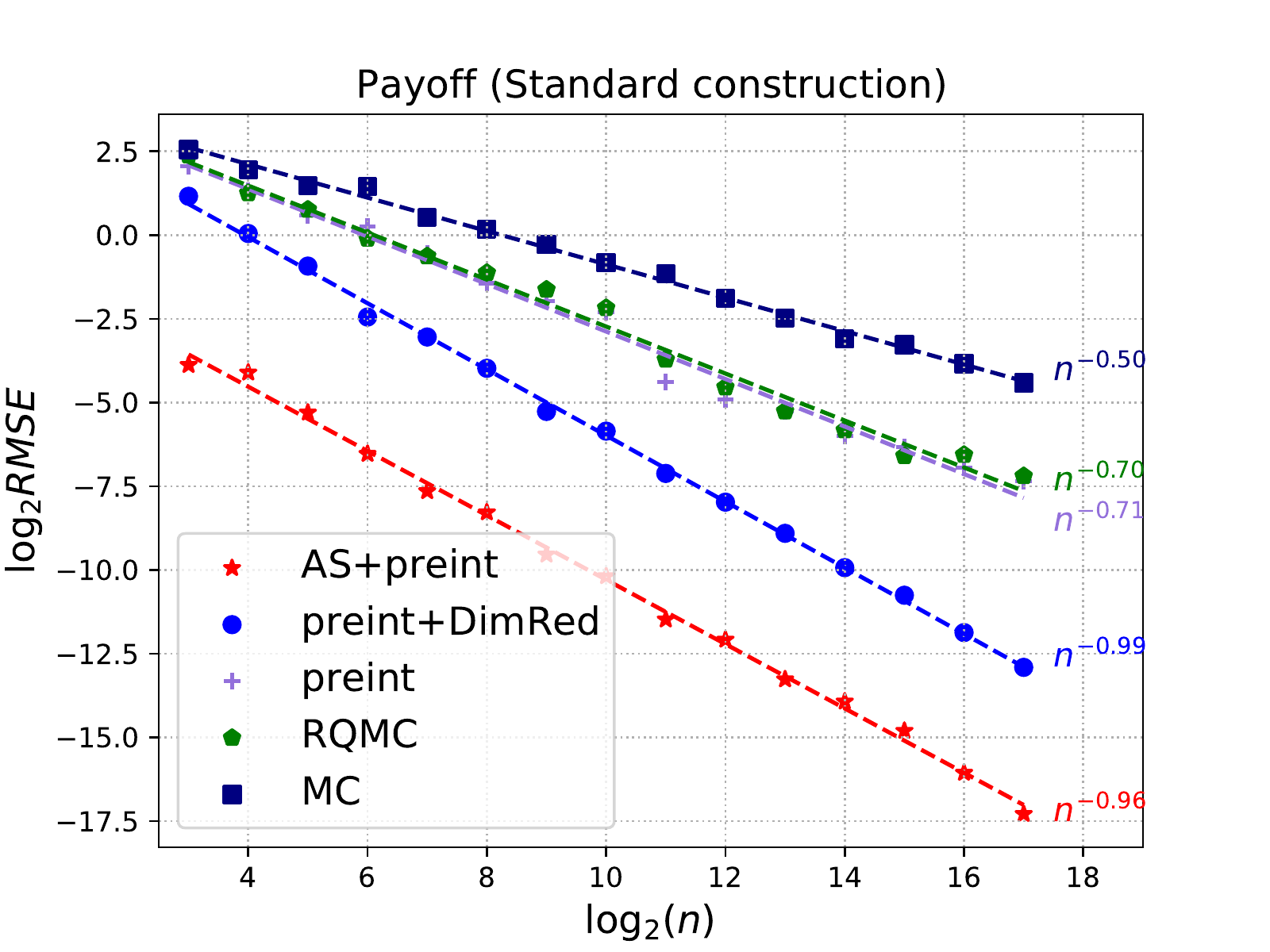}
\end{subfigure}
\begin{subfigure}{.5\textwidth}
\includegraphics[width=\textwidth]{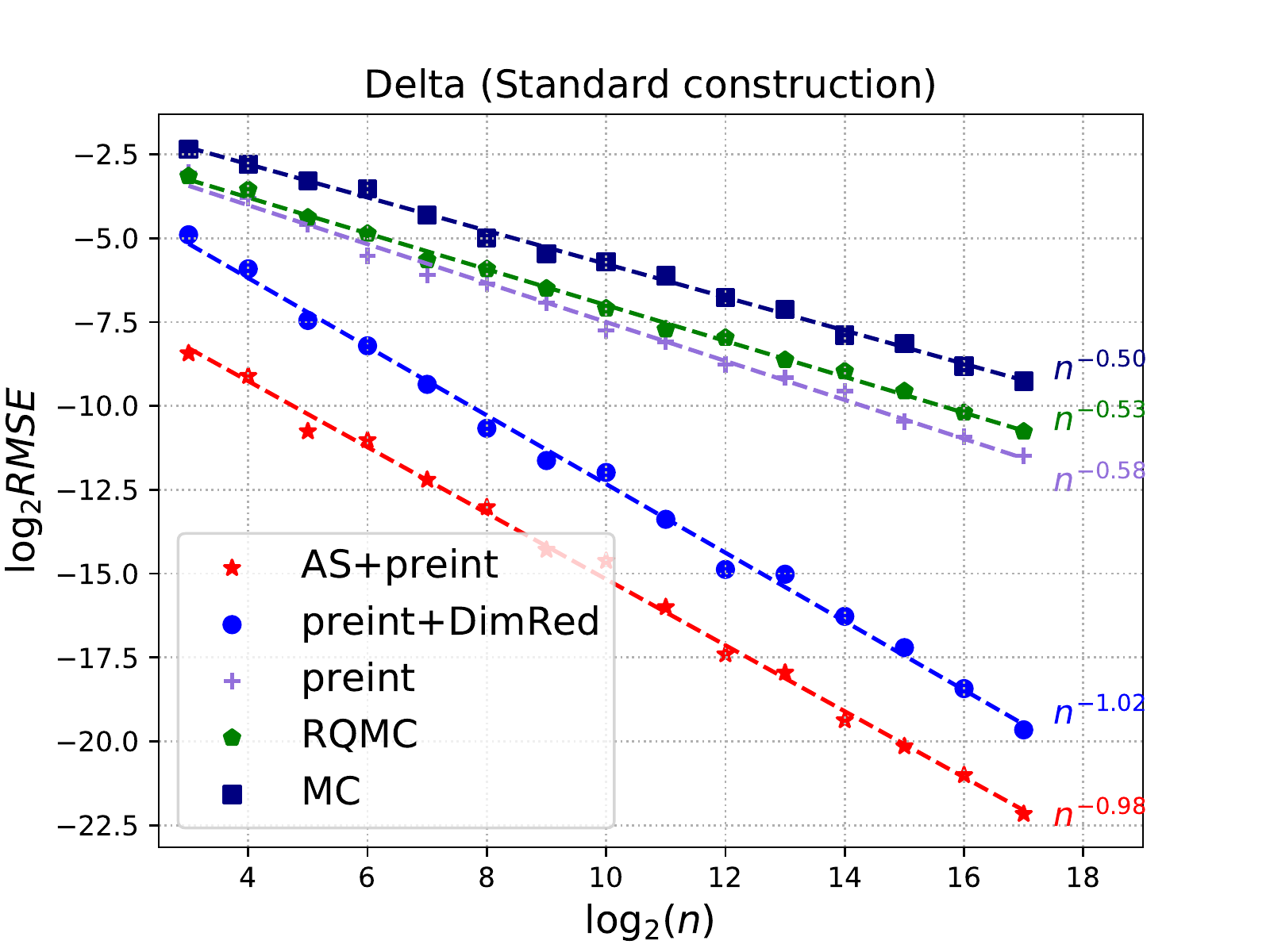}
\end{subfigure}
\begin{subfigure}{.5\textwidth}
\includegraphics[width=\textwidth]{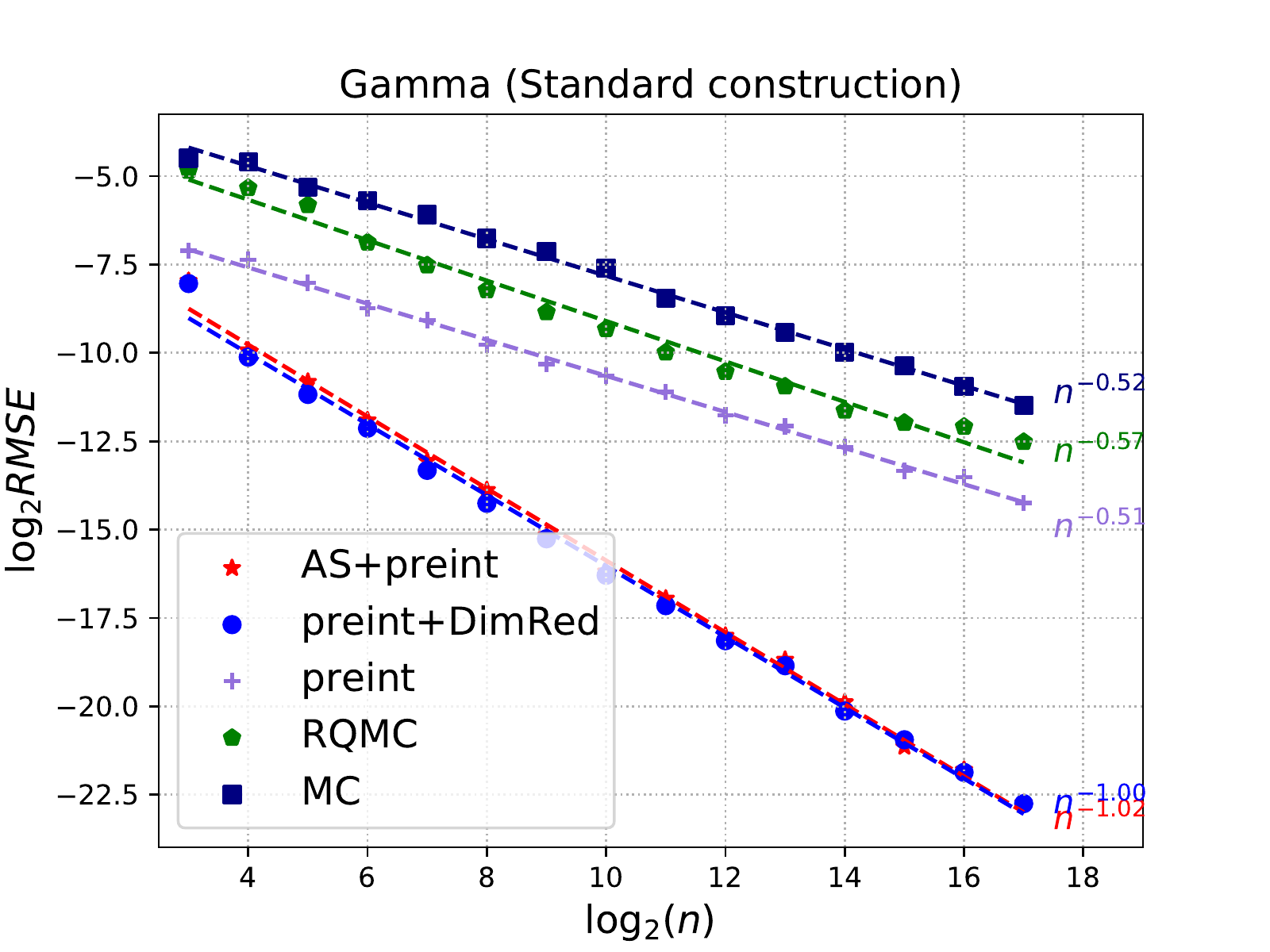}
\end{subfigure}
\begin{subfigure}{.5\textwidth}
\includegraphics[width=\textwidth]{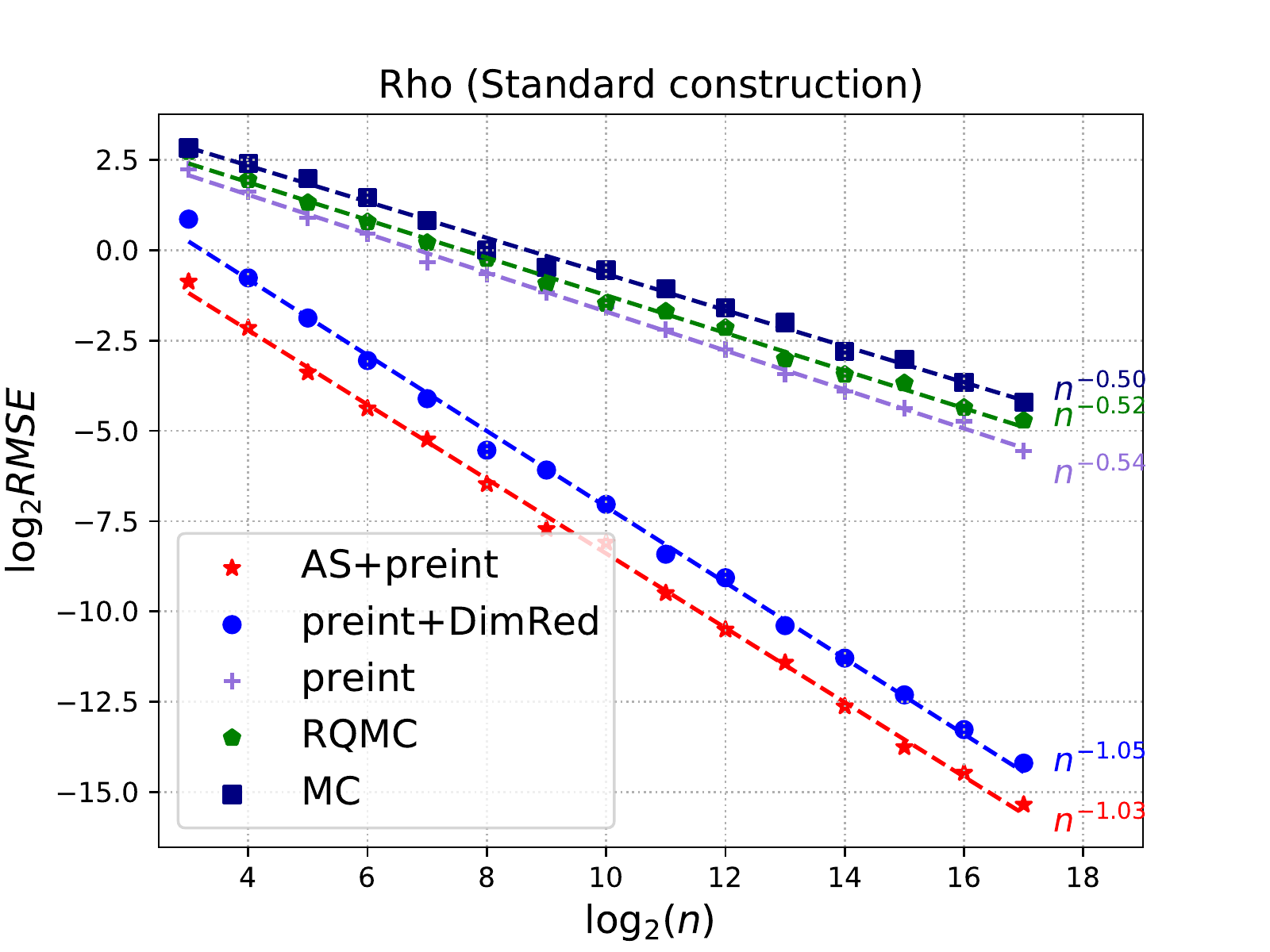}
\end{subfigure}
\begin{subfigure}{.5\textwidth}
\includegraphics[width=\textwidth]{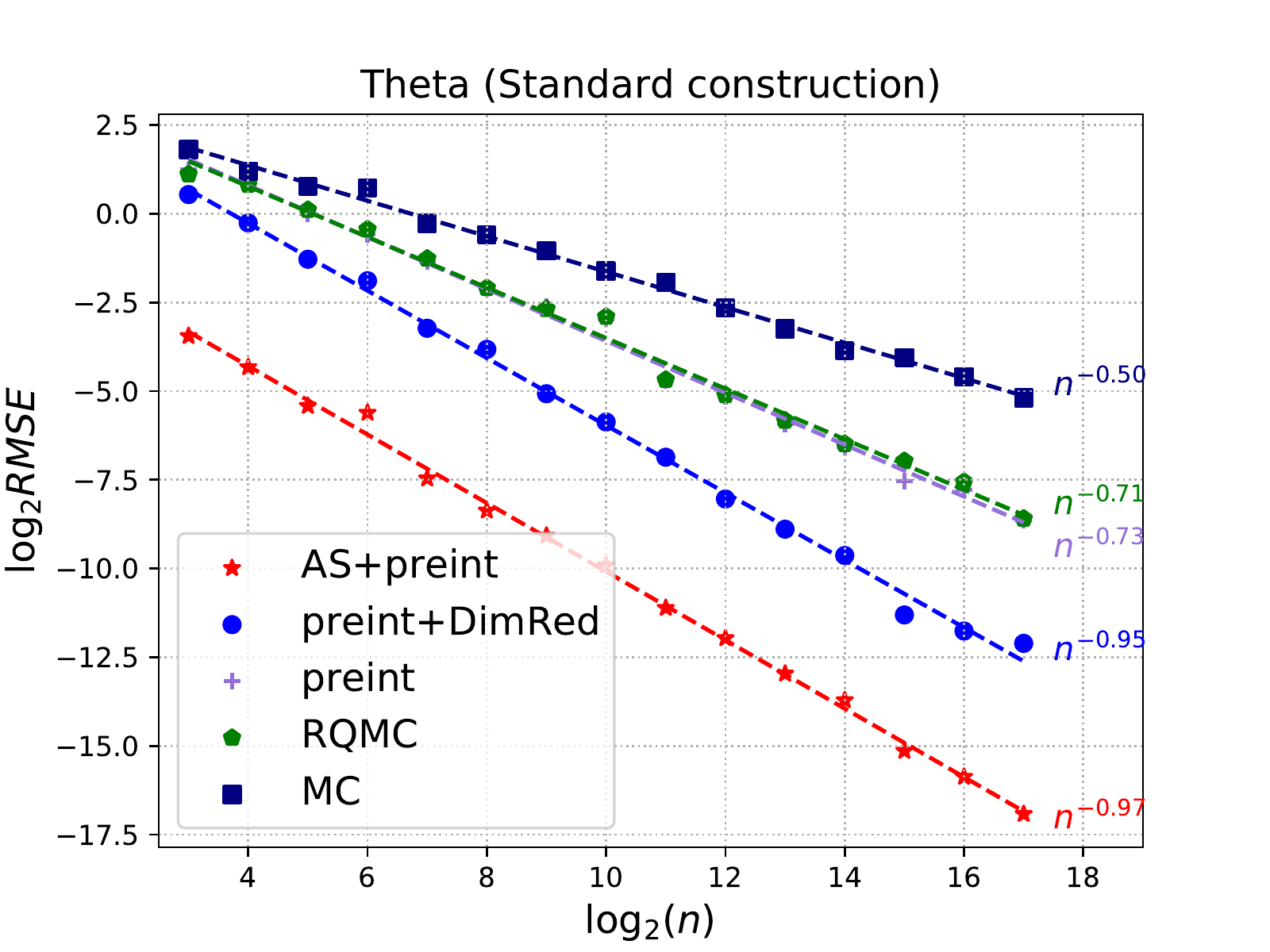}
\end{subfigure}
\begin{subfigure}{.5\textwidth}
\includegraphics[width=\textwidth]{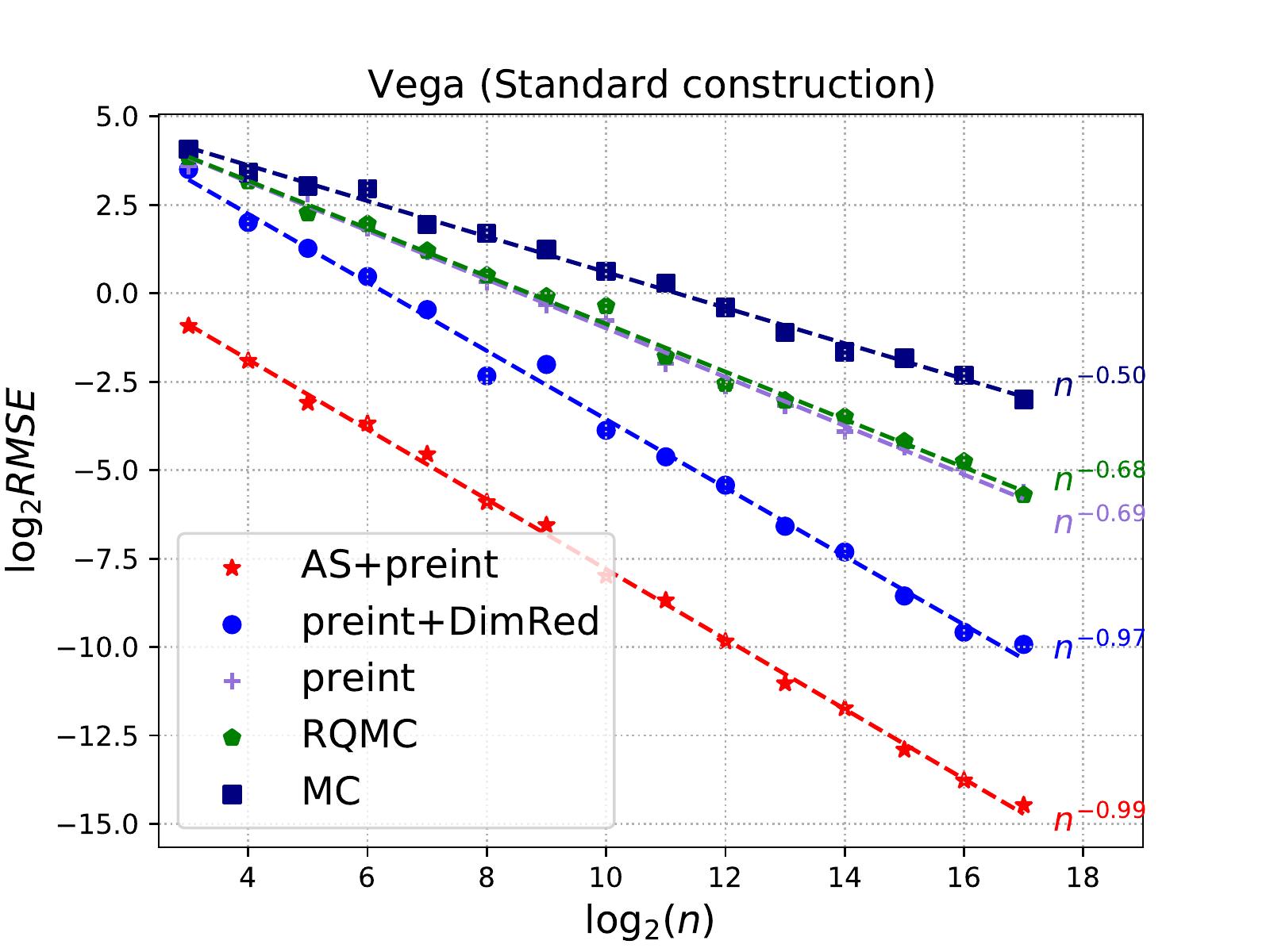}
\end{subfigure}
\caption{Single asset option. Standard construction of Brownian motion with $d=50$. This and subsequent figures include
least squares estimated slopes on the log--log scale.}
\label{fig: std}
\end{figure}

\begin{figure}
\begin{subfigure}{.5\textwidth}
\includegraphics[width=\textwidth]{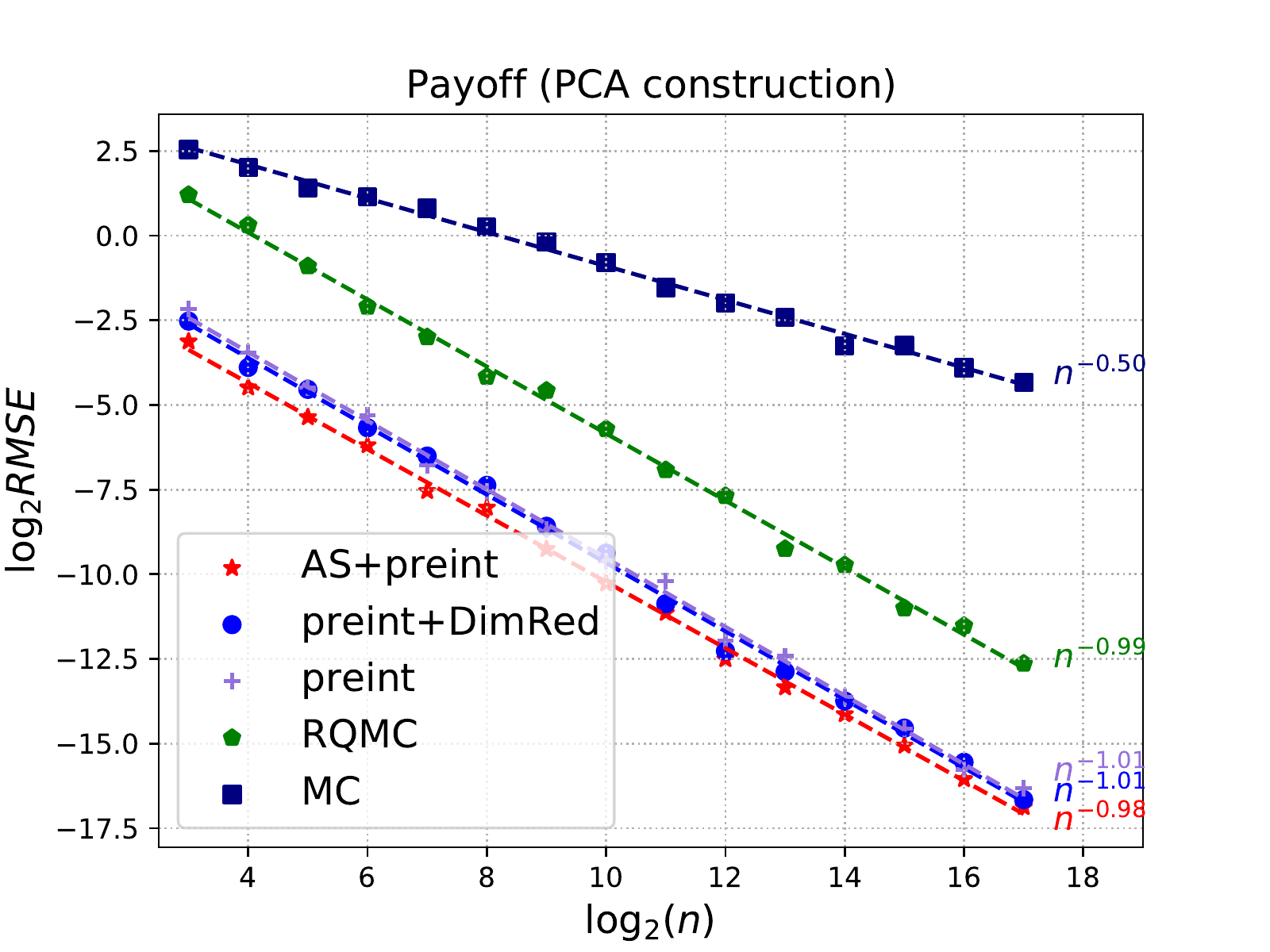}
\end{subfigure}
\begin{subfigure}{.5\textwidth}
\includegraphics[width=\textwidth]{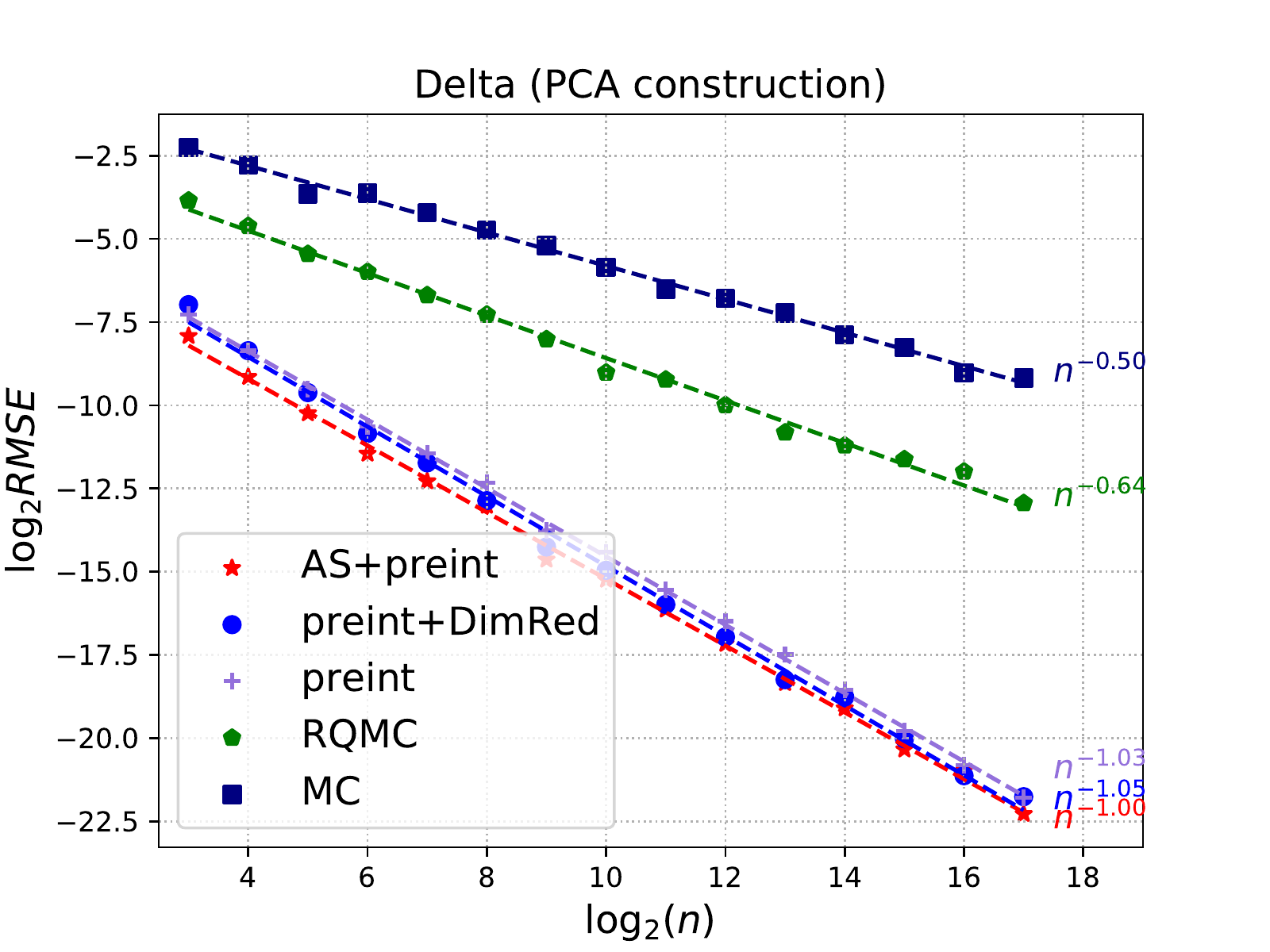}
\end{subfigure}
\begin{subfigure}{.5\textwidth}
\includegraphics[width=\textwidth]{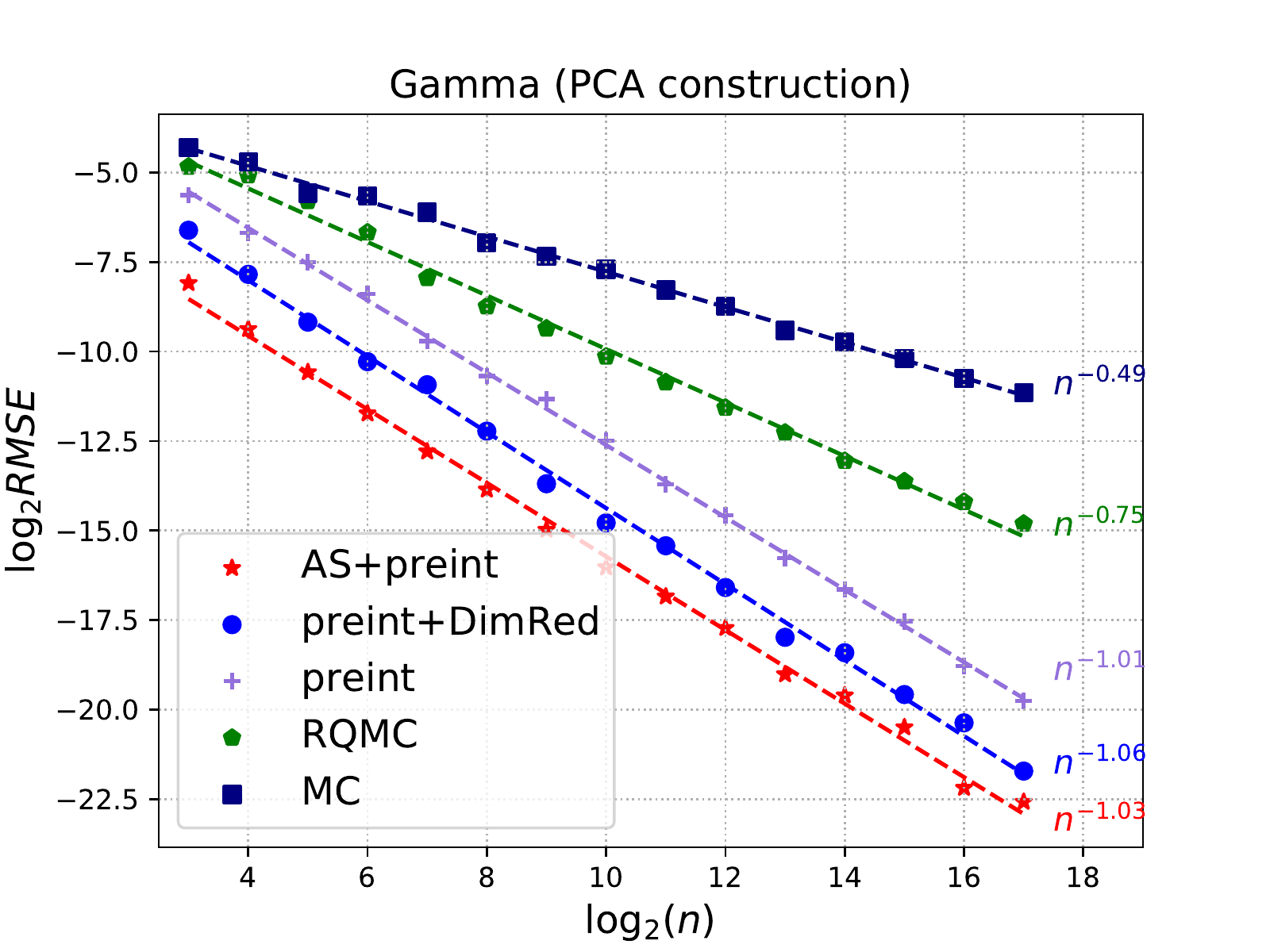}
\end{subfigure}
\begin{subfigure}{.5\textwidth}
\includegraphics[width=\textwidth]{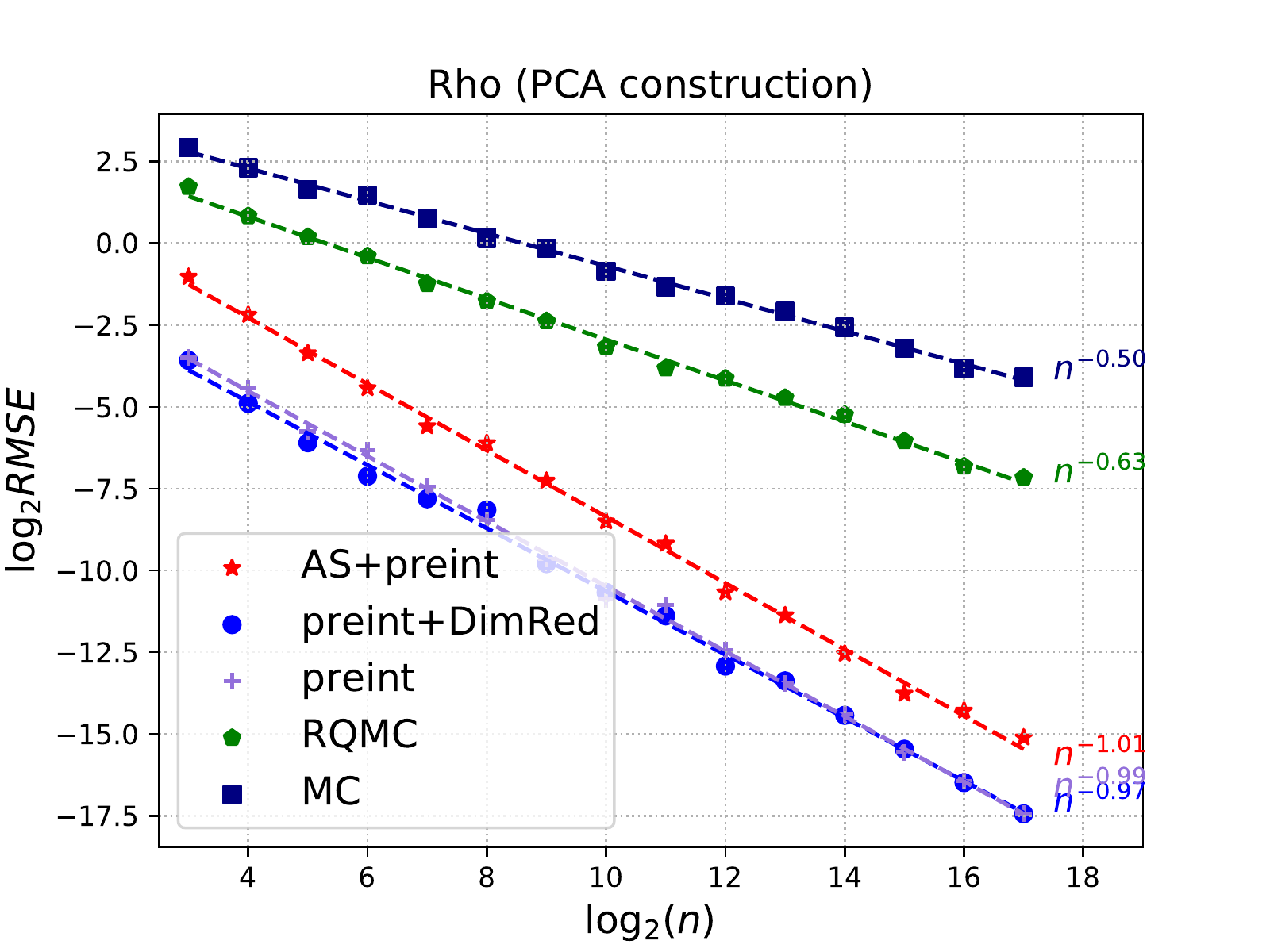}
\end{subfigure}
\begin{subfigure}{.5\textwidth}
\includegraphics[width=\textwidth]{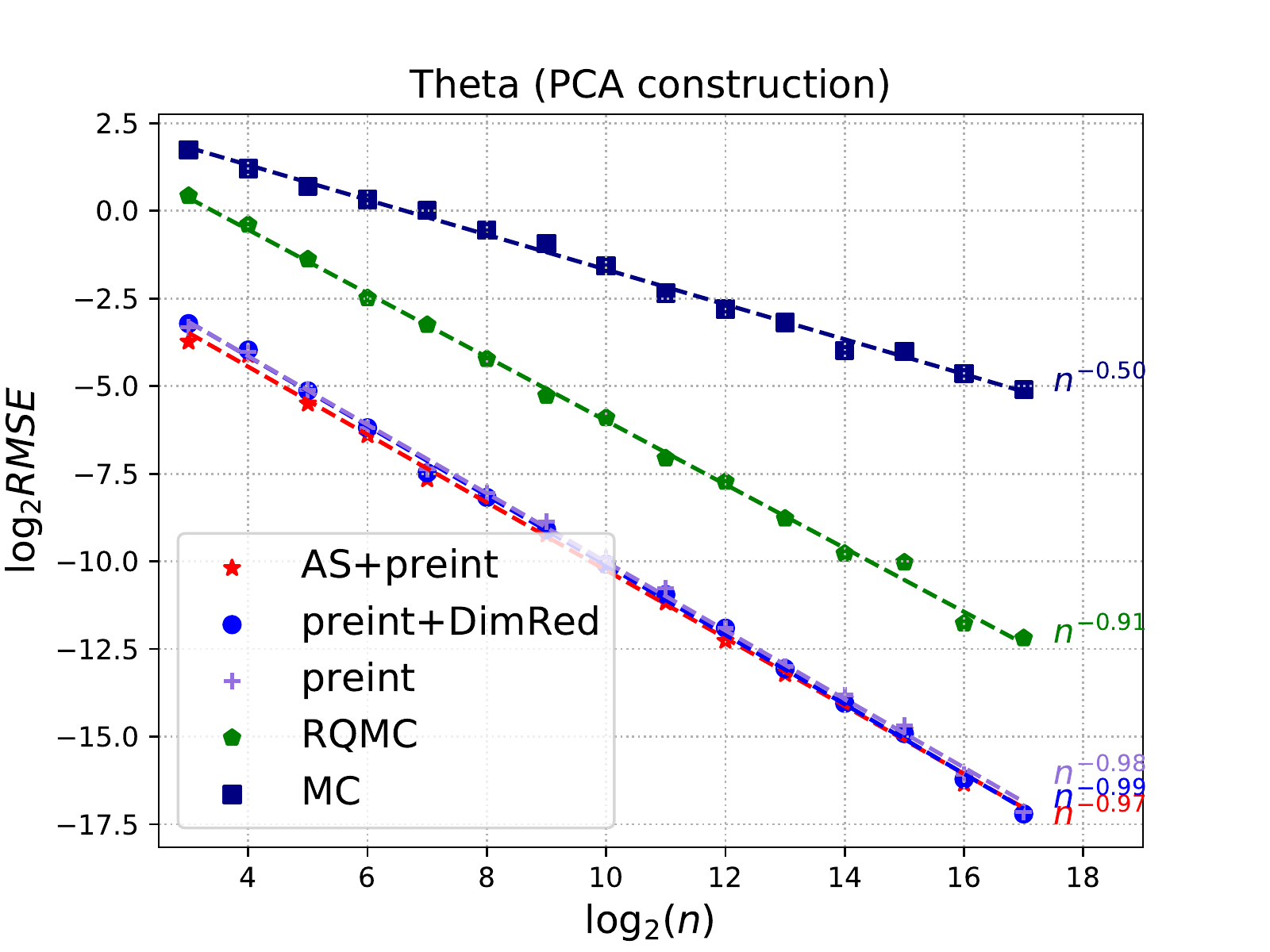}
\end{subfigure}
\begin{subfigure}{.5\textwidth}
\includegraphics[width=\textwidth]{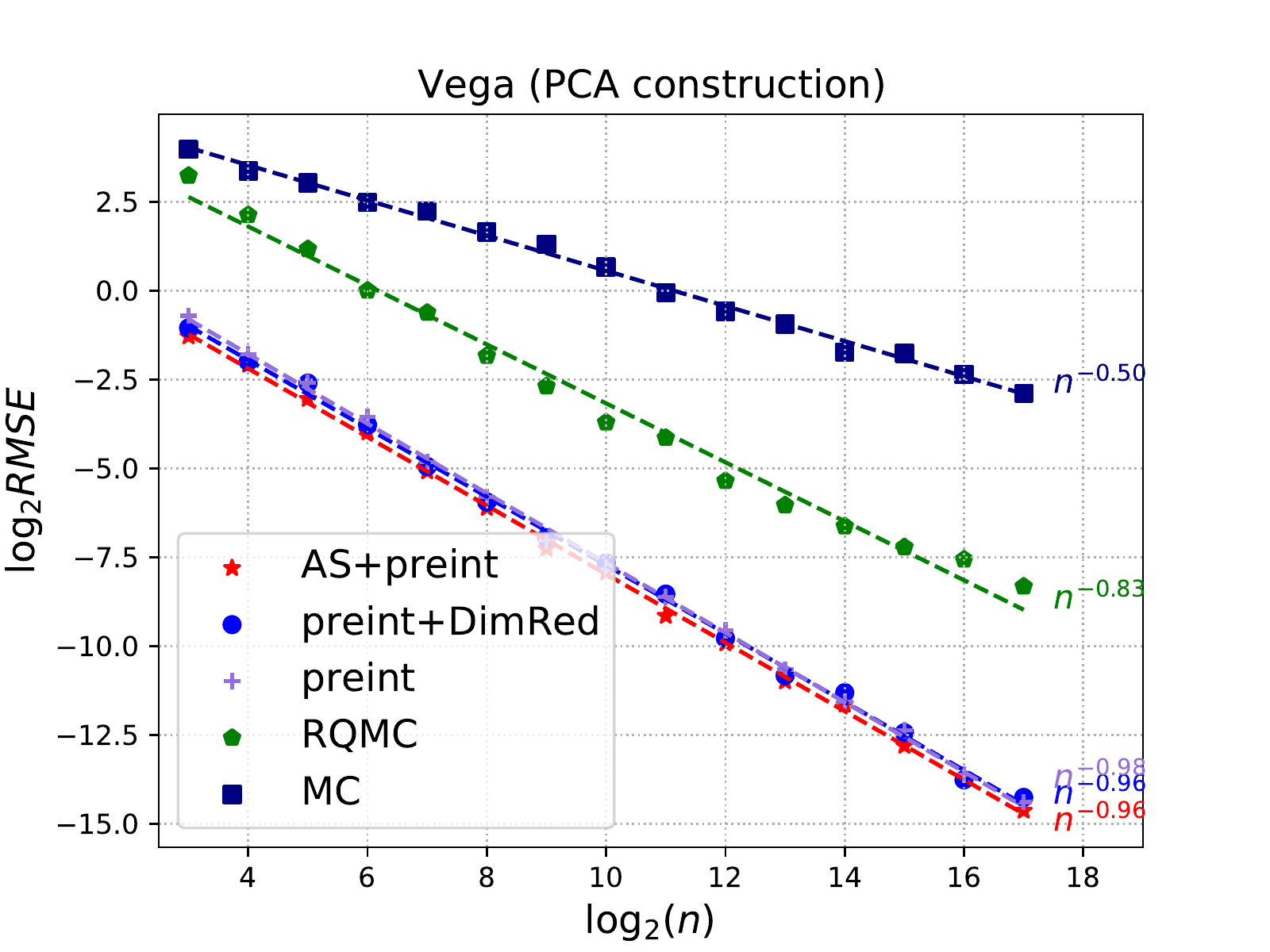}
\end{subfigure}
\caption{Single asset option. PCA construction of Brownian motion with $d=50$.}
\label{fig: pca}
\end{figure}

\subsection{Basket option}
A basket option depends on the weighted average of several assets.
Suppose that the $L$ assets $S^{(\ell)},\ldots,S^{(L)}$ follow the SDE
\begin{align*}
\mrd S_t^{(\ell)}=r_\ell S_t^{(\ell)}\mrd t+\sigma_\ell S_t^{(\ell)}\mrd B_t^{(\ell)},
\end{align*}
where $\{B^{(\ell)}\}_{1\leq \ell\leq L}$ are standard Brownian motions with correlation
\begin{align*}
\mathrm{Corr}(B_t^{(\ell)},B_t^{(k)})=\rho_{\ell k}
\end{align*}
for all $t>0$.
For some nonnegative weights $w_1+\ldots+w_L=1$, the payoff function of the Asian basket call option is given by
\begin{align*}
\Biggl(\,\sum_{\ell=1}^Lw_\ell \bar S^{(\ell)} -K\Biggr)_+
\end{align*}
where $\bar S^{(\ell)}$ is the arithmetic average of $S^{(\ell)}_t$ in the time interval $[0,T]$. Here, we only consider $L=2$ assets. To generate $B^{(1)},B^{(2)}$ with correlation $\rho$, we can generate two independent standard Brownian motions $W^{(1)},W^{(2)}$
letting
\begin{align*}
B^{(1)}=W^{(1)}\quad\text{and}\quad B^{(2)}=\rho W^{(1)}+\sqrt{1-\rho^2}W^{(2)}.
\end{align*}
Following the same discretization as before, we can generate
$(\bsz\tran,\tilde\bsz\tran)\sim\dnorm(0,I_{2d})$. Then
for time steps $j=1,\dots,d$, let
\begin{align*}
S_j ^{(1)}&=S_0^{(1)}\exp\Bigl(\Bigl(r_1-\frac{\sigma_1^{2}}{2}\Bigr)j\Delta t+\sigma_1\sum_{k=1}^d R_{jk}z_{k}\Bigr),\quad\text{and}\\
S_j^{(2)}&=S_0^{(2)}\exp\Bigl(\Bigl(r_2-\frac{\sigma_2^{2}}{2}\Bigr)j\Delta t+\sigma_2\Bigl(\rho\sum_{k=1}^d R_{jk}z_{k}+\sqrt{1-\rho^2}\sum_{k=1}^dR_{jk}\tilde z_{k}\Bigr)\Bigr).
\end{align*}
Again, the $d\times d$ matrix $R$ can be constructed by
the standard construction or the PCA construction.
Thus, the expected payoff can be written as
\begin{align*}
\e\Bigl(\Bigl(\frac{w_1}{d}\sum_{j=1}^d S_j^{(1)} + \frac{w_2}{d}\sum_{j=1}^d S_j^{(2)}-K\Bigr)_+ \Bigr),
\end{align*}
where the expectation is taken over
$(\bsz\tran,\tilde\bsz\tran)\tran\sim\dnorm(0,I_{2d})$.

In the pre-integration step, we choose to integrate out $z_{1}$.
This can be easily carried out as in equations~\eqref{equ: S=K}
and~\eqref{equ: pre-int} provided that the first column of $\tilde R$ is nonnegative.
We take $d=50$, $T=1$, $\rho=0.5$, $S_0^{(1)}=S_0^{(2)}=100$ and $K=95$.
The RMSE values are plotted in Figure~\ref{fig: basket}.
In the left panel, we take $r_1=0.1$, $r_2=0.2$,
$\sigma_1=0.2$, $\sigma_2=0.4$, $w_1=0.8$ and $w_2=0.2$. In the right panel, we take $r_1=0.2$, $r_2=0.1$, $\sigma_1=0.4$, $\sigma_2=0.2$, $w_1=0.2$ and $w_8=0.8$.
This reverses the roles of the two assets which will make
a difference when one pre-integrates over the first of the $2d$
inputs. We use $M=128$ as before.
In the top row of Figure~\ref{fig: basket}, the matrix $R$ is obtained by the standard construction, while in the bottom row, $R$ is obtained by the PCA construction.

\begin{figure}
\begin{subfigure}{.49\textwidth}
\includegraphics[width=\textwidth]{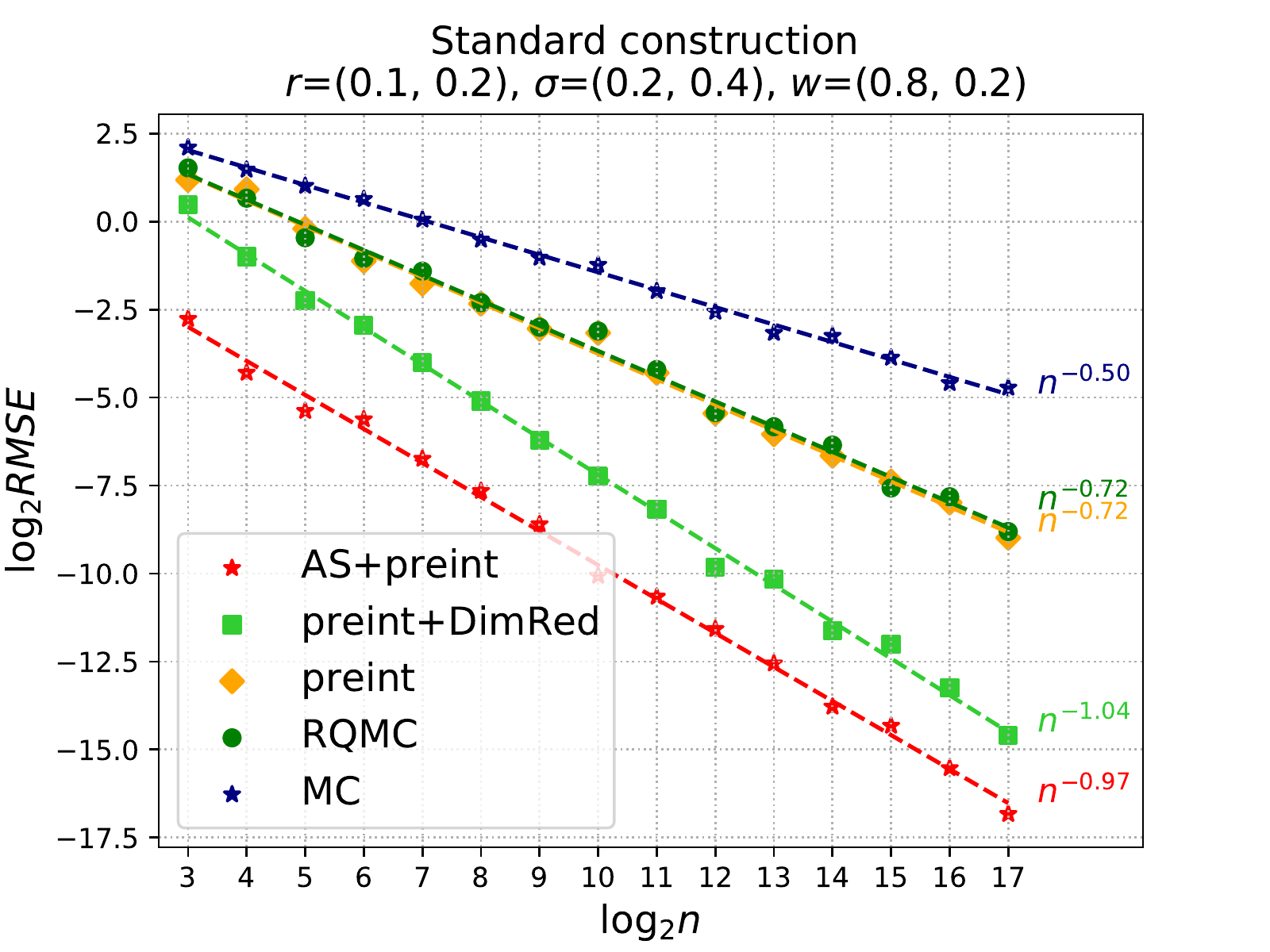}
\end{subfigure}
\begin{subfigure}{.49\textwidth}
\includegraphics[width=\textwidth]{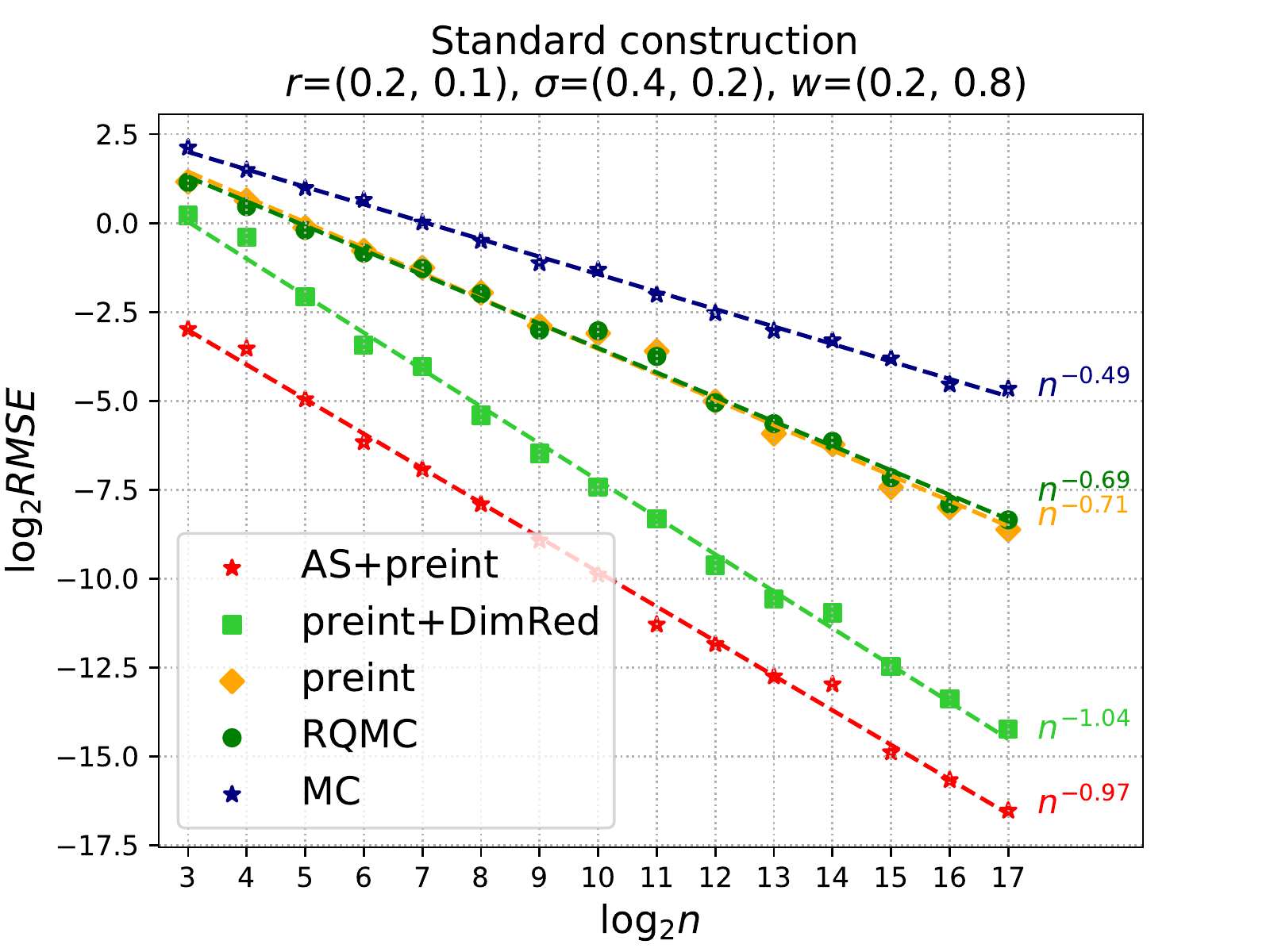}
\end{subfigure}
\begin{subfigure}{.49\textwidth}
\includegraphics[width=\textwidth]{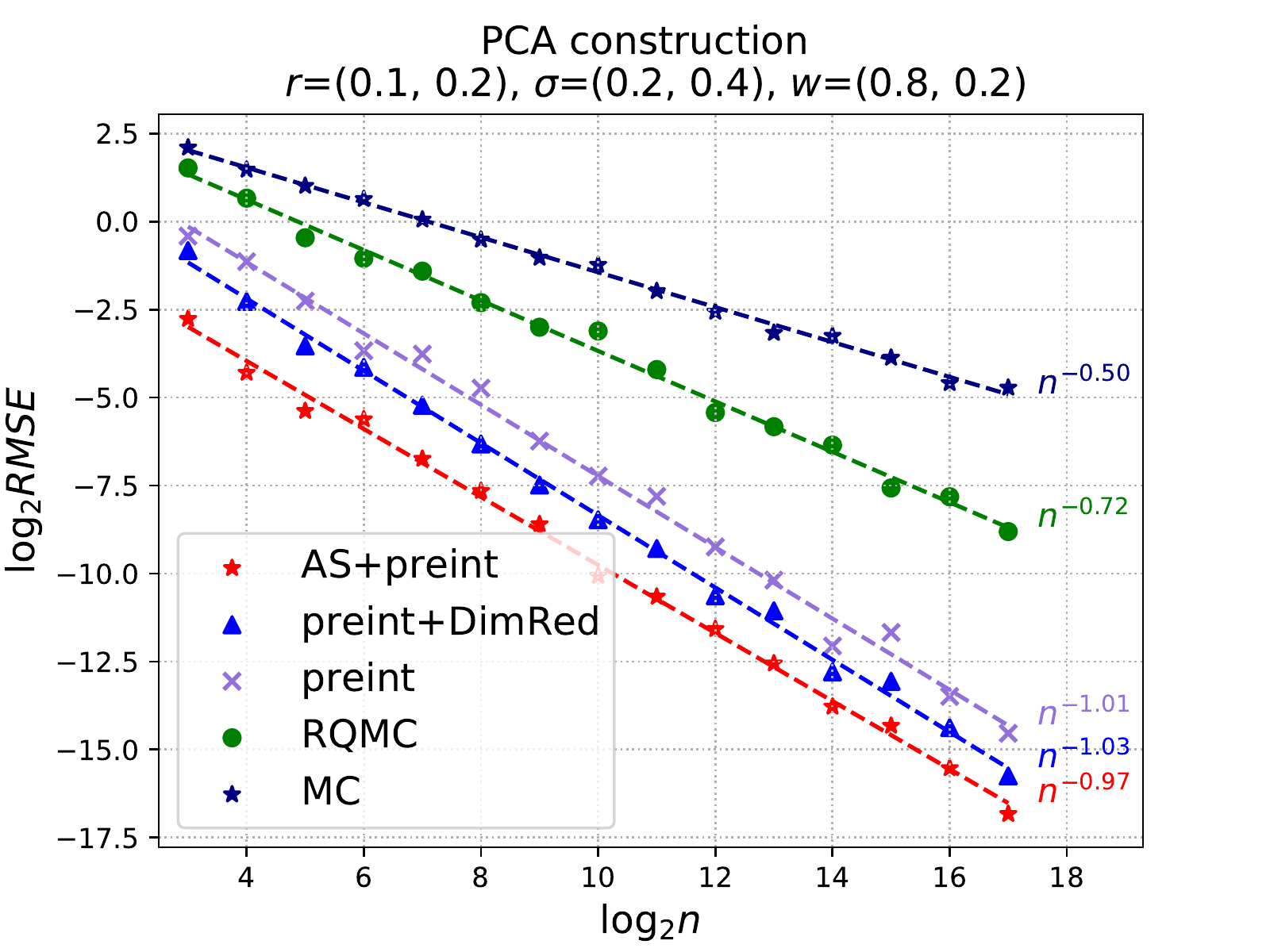}
\end{subfigure}
\begin{subfigure}{.49\textwidth}
\includegraphics[width=\textwidth]{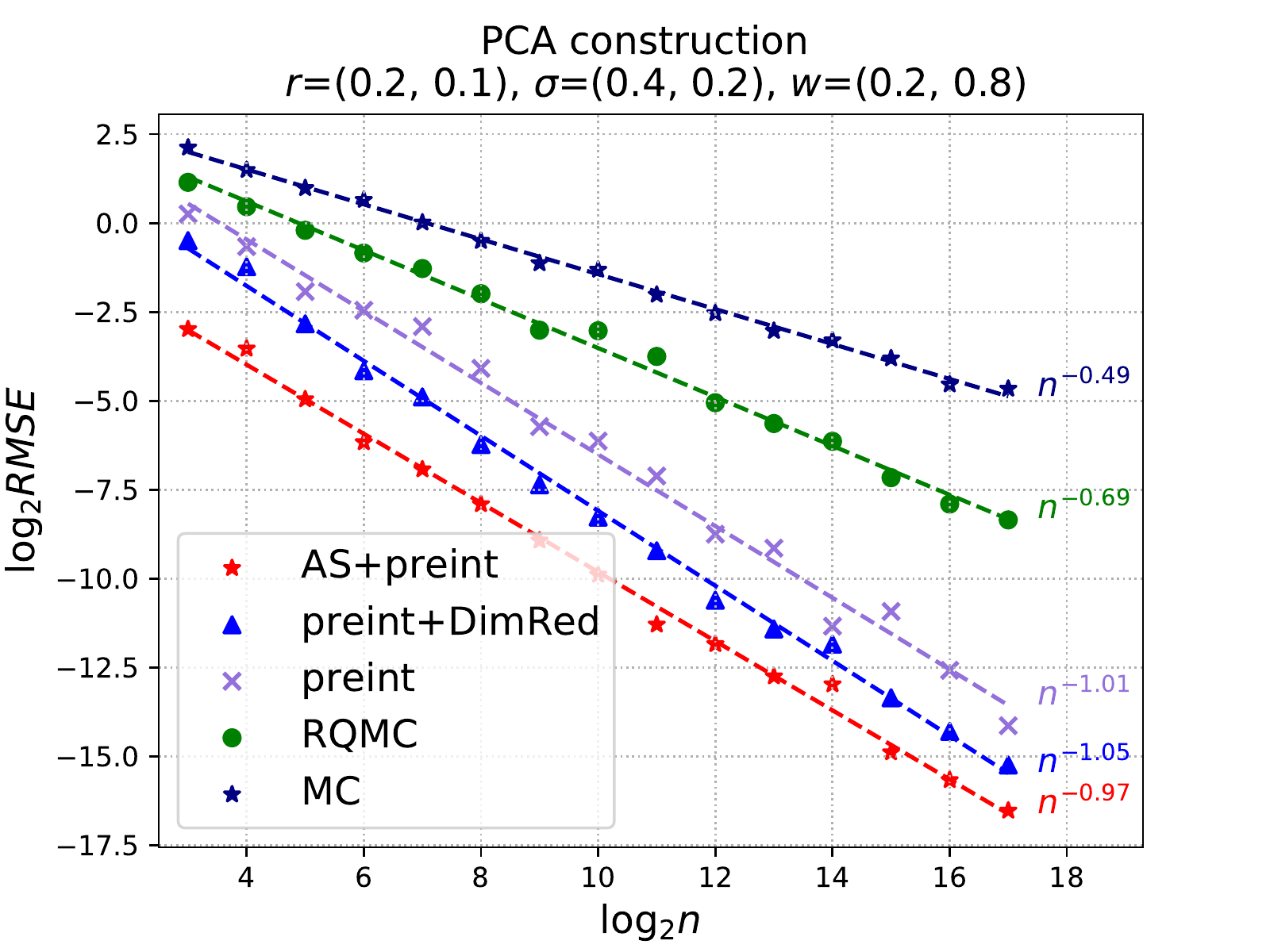}
\end{subfigure}
\caption{Basket option pricing.}
\label{fig: basket}
\end{figure}

A few observations are in order:
\begin{enumerate}[(a)]
\item For the standard construction, pre-integrating over $z_1$ brings little improvement over plain RQMC. But for PCA construction, pre-integrating over $z_1$ brings a big variance reduction.

\item The dimension reduction technique from \cite{xiao2018conditional} largely improves the RMSE from pre-integration without dimension reduction. This improvement is particularly significant for the standard construction.

\item Active subspace pre-integration, \texttt{AS+preint}, has the best performance for both the standard and PCA constructions. It is even better than pre-integrating out the first principal component of Brownian motion with dimension reduction. In this example, the active subspace method is able to find a better direction than the first principal component over which to pre-integrate.
\end{enumerate}

This problem required $L=2$ Brownian motions and
the above examples used the same decomposition
to sample them both.  A sharper principal components
analysis would merge the two Brownian motions into
a single $2d$-dimensional process and use the
principal components from their joint covariance
matrix.  We call this the full PCA construction
as described next.

The processes $B^{(1)}=\sigma_1R\bsz$,
and $B{(2)}=\sigma_2R(\rho \bsz+\sqrt{1-\rho^2}\tilde \bsz)$,
have joint distribution
\begin{align*}
\begin{pmatrix}
B^{(1)}\\ B^{(2)}
\end{pmatrix}\sim
\dnorm\left(0, \begin{pmatrix}
\sigma_1^2\Sigma & \rho\sigma_1\sigma_2\Sigma\\
\rho\sigma_1\sigma_2\Sigma & \sigma_2^2\Sigma
\end{pmatrix}\right).
\end{align*}
Let $\wt\Sigma$ be the joint covariance matrix above.
An alternative generation method is to pick $\tilde R$
with $\tilde R\tilde R\tran=\tilde\Sigma$, and let
\begin{align*}
\begin{pmatrix}
B^{(1)}\\ B^{(2)}
\end{pmatrix}
=\tilde R
\begin{pmatrix}
\bsz\\
\tilde\bsz
\end{pmatrix}
,\quad\text{ where }\quad
\begin{pmatrix}
\bsz\\
\tilde\bsz
\end{pmatrix}
\sim\dnorm(0,I_{2d}).
\end{align*}
The matrix $\tilde R$ can be found by either a
Cholesky decomposition or eigendecomposition of $\tilde R$.
We call this method full standard construction or full PCA construction.
Taking $B^{(1)}=\sigma_1 R\bsz$,
and $B^{(2)}=\sigma_2R(\rho \bsz+\sqrt{1-\rho^2}\tilde\bsz)$, it is equivalent to taking
\[
\tilde R=\begin{pmatrix}
\sigma_1 I_d & \mathbf{0} \\ \sigma_2\rho I_d & \sqrt{1-\rho^2}\sigma_2I_d
\end{pmatrix}
\begin{pmatrix}
R & \mathbf{0} \\ \mathbf{0} & R
\end{pmatrix}.
\]
We call this the ordinary PCA or standard construction depending on whether $R$ is from the PCA or standard construction.

Figure~\ref{fig: basket compare}
compares results using this full PCA construction.
All methods apply pre-integration before using RQMC.
For active subspace pre-integration,
we pre-integrate along the direction in $\real^{2d}$
found by the active subspace method.
For ``full PCA", we pre-integrate along the first principal component
of $\tilde \Sigma$. 
We can see that active subspace pre-integration has a better RMSE
than even the full PCA construction with dimension reduction, which we consider to be the
natural generalization of the PCA construction to this setting.

\begin{figure}
\begin{subfigure}{.49\textwidth}
\includegraphics[width=\textwidth]{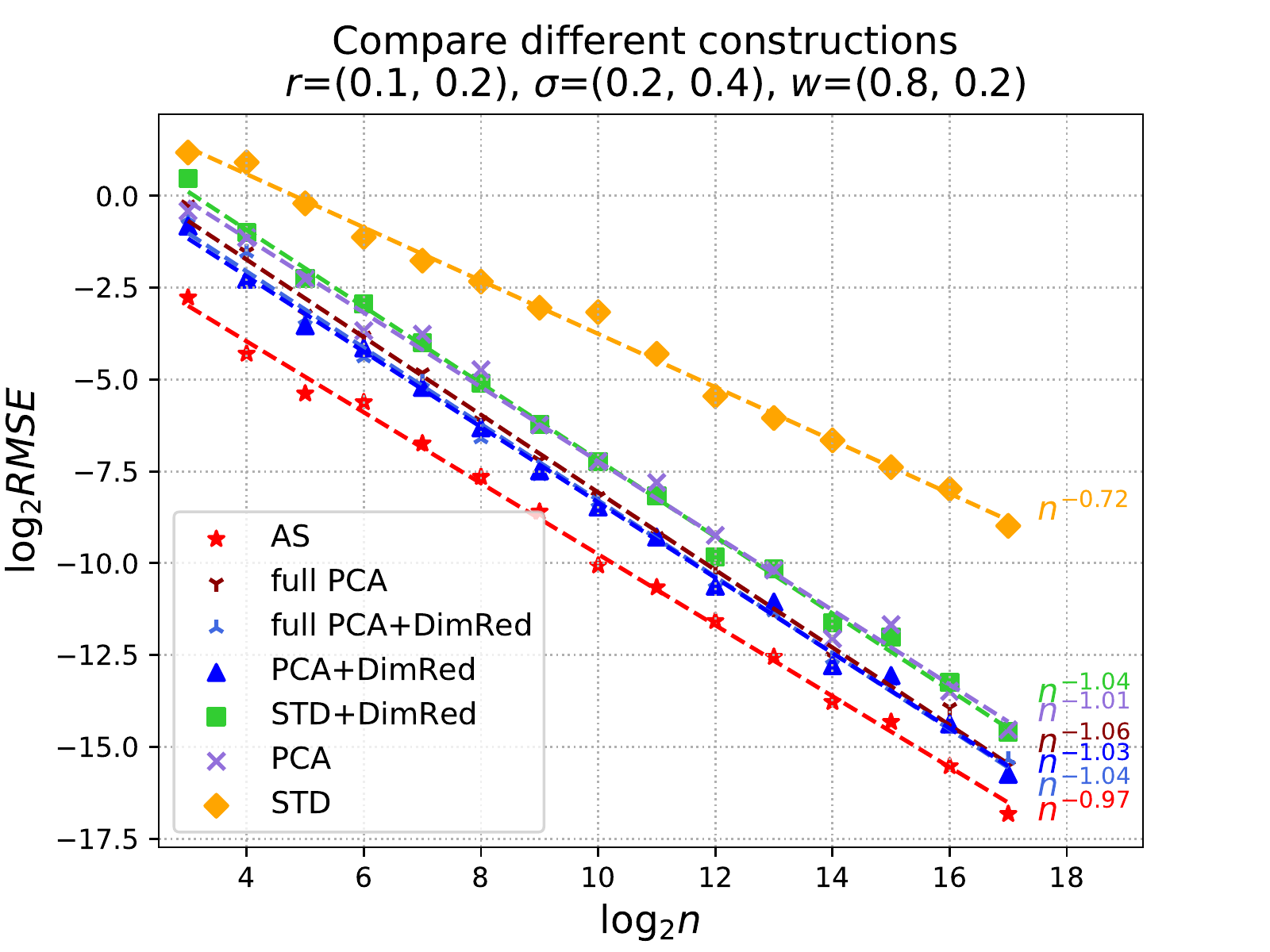}
\end{subfigure}
\begin{subfigure}{.49\textwidth}
\includegraphics[width=\textwidth]{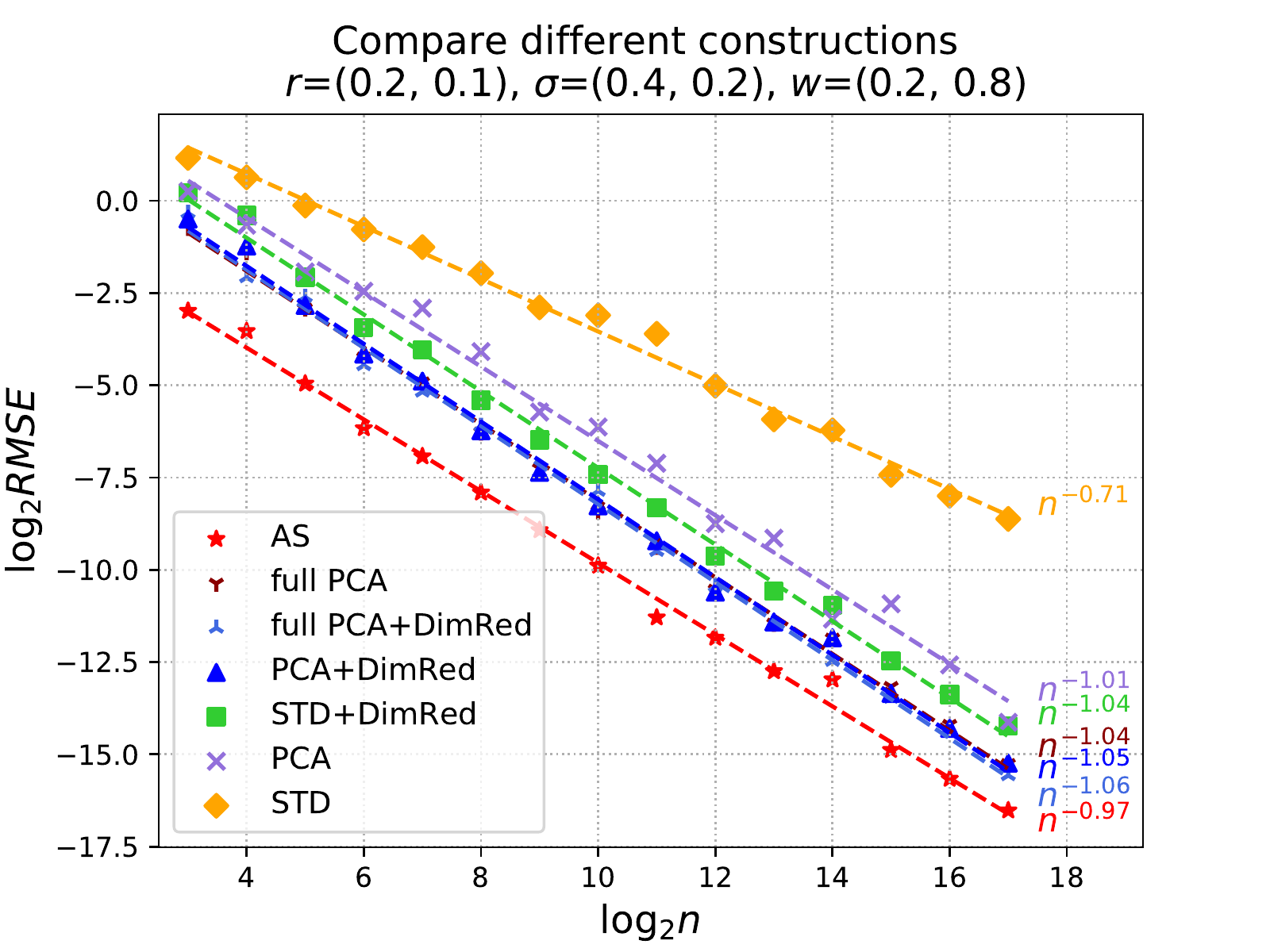}
\end{subfigure}
\caption{
These plots compare pre-integration strategies for
two basket options. We compare active subspace pre-integration
to strategies with full and partial PCA and standard contructions
of Brownian motion.
}
\label{fig: basket compare}
\end{figure}

\subsection{Timing}

Pre-integration changes the computational cost of
RQMC and this has not been much studied in
the literature.  Our figures compare the RMSE
of different samplers as a function of the number $n$ of evaluations.
Efficiency depends also on running times.
Here we make study of running times for
the pre-integration methods we studied.
It is brief in order to conserve space.
Running times depend on implementation details
that could make our efficiency calculations differ from others'.
All of timings we report were conducted on a MacBook Pro with 8 cores and 16GB memory. 
We simulated them all with $10$ replicates
and $n=2^{15}$.  The standard errors of the average evaluation times were negligible, about $0.3$\%--$0.8$\% of the corresponding means. We also computed times to find $\widehat C$ and $\wh{\widetilde C}$ (defined in equations \eqref{equ: Sigma hat} and \eqref{equ: hat tilde C}).
Those had small standard errors too, and their
costs are a small fraction of the total.
The costs of finding the eigendecompositons are
negligible in our examples.

Table~\ref{tab:alltimes} shows the results of
our timings.  For the 6 integrands, we find
that pre-integration raises the cost of
computation by roughly $12$ to $16$-fold.
That extra cost is not enough to offset
the gains from pre-integration with large $n$
except for pre-integration of the first component
in the standard construction.  That variable is not
very important and we might well have expected pre-integration
to be unhelpful for it.

Most of the pre-integrated computations took
about the same amount of time but
plain pre-integration and pre-integration
with dimension reduction take slighly
more time for the standard construction.
Upon investigation we found that those methods
took slightly more Newton iterations on average.
In hindsight this makes sense.  The Newton
iterations started at $0$.  In the standard construction,
the first variable does not have a very strong effect
on the option value, requiring it to take larger
values to reach the positivity threshold $x_1=\gamma(\bsx_{2{:}d})$ and hence
(a few) more iterations.

Another component of cost for active subspace methods
is in computing an approximation to $C$ and $\widetilde C$.  We used $M=128$
function evaluations.  Those are about $2d$ times as expensive
as evaluations because they use divided differences.
One advantage of active subspace pre-integration is that
it uses divided differences of the original integrand.
Pre-integration plus dimension reduction requires
divided differences of the pre-integrated function
with associated Newton searches, and that costs more.


\begin{table}
\centering
\begin{tabular}{lccccccc}
\toprule
 & MC & RQMC & pre-int & AS+pre-int & pre-int+DimRed & $\wh C$ & $\wh{\widetilde C}$ \\
\midrule
\multirow{2}{*}{Payoff}
& 0.6 &        0.6 &          \phz7.5 &             \phz6.9 &                 \phz7.4 &       0.3 &            2.9 \\
& 0.6 &        0.6 &          \phz 6.9 &             \phz 6.9 &               \phz  6.9 &        0.3 &            2.7\\
\midrule
\multirow{2}{*}{Delta}
&0.5 &        0.4 &          \phz 5.7 &             \phz 5.2 &                 \phz 5.7 &        0.2 &            2.2  \\
& 0.5 &        0.4 &          \phz 5.1 &             \phz 5.2 &                 \phz 5.2 &        0.2 &            2.0 \\
\midrule
\multirow{2}{*}{Gamma}
&      0.6 &        0.6 &         12.2 &            12.2 &                12.3 &        0.2 &            4.7 \\
& 0.6 &        0.6 &         11.7 &            12.3 &                11.7 &        0.2 &            4.5 \\
\midrule
\multirow{2}{*}{Rho}
&  0.9 &        0.9 &         10.6 &            10.1 &                10.6 &        0.3 &            4.1 \\
& 0.9 &        0.9 &         10.1 &            10.2 &                10.1 &        0.4 &            3.9\\
\midrule
\multirow{2}{*}{Theta}
&  1.0 &        1.0 &         13.9 &            13.3 &                14.0 &        0.4 &            5.4  \\
&  1.0 &        1.0 &         13.4 &            13.4 &                13.4 &        0.4 &            5.2 \\
\midrule
\multirow{2}{*}{Vega}
& 0.6 &        0.6 &          \phz 9.7 &             \phz 9.1 &                 \phz 9.7 &        0.3 &            3.7\\
& 0.6 &        0.6 &          \phz 9.1 &             \phz 9.1 &                 \phz 9.1 &        0.3 &            3.5  \\
\bottomrule
\end{tabular}
\caption{\label{tab:alltimes}
Time in seconds to compute all Asian option integrands and methods. All simulations take $2^{15}$ samples and the times are averaged over 10 replicates.
Each integrand has two rows, where the top row uses the standard construction and the bottom row uses
the PCA construction.
The right two columns are the times used to estimate $\wh C$ and $\wh{\widetilde C}$
 by $M=128$ samples.}
\end{table}

\section{Discussion}\label{sec:discussion}

In this paper we have studied a kind of conditional RQMC known as pre-integration.
We found that, just like conditional MC, the procedure can reduce variance but cannot
increase it.  We proposed to pre-integrate over the first component in an active subspace approximation,
which is also known as the gradient PCA approximation.  We showed a close relationship between
this choice of pre-integration variable and what one would get using a computationally infeasible
but well motivated choice by maximizing the Sobol' index of a linear combination of variables.

In the numerical examples of option pricing, we see that active subspace pre-integration achieves a better RMSE than previous methods when using the standard construction of Brownian motion. For the PCA construction, the proposed method has comparable accuracy in four of six cases, is better once and is worse once.
For those six integrands, the PCA construction is already very good.
Active subspace pre-integration still provides an automatic way to choose the
pre-integration direction. Even using the standard construction it is almost
as good as pre-integration with the PCA construction.
It can be used in settings where there is no strong incumbent decomposition
analogous to the PCA for the Asian option.
We saw it perform well for basket options.
We even saw an improvement for Gamma
in the very well studied case of the Asian option with the PCA construction.

We note that active subspaces use an uncentered PCA analysis of the
matrix of sample gradients.  One could use instead a centered analysis
of $\e( (\nabla f(\bsx)-\eta) (\nabla f(\bsx)-\eta)\tran)$ where
$\eta = \e(\nabla f(\bsx))$.
The potential advantage of this is that $\nabla f(\bsx)-\eta$
is the gradient of $f(\bsx)-\eta\tran\bsx$ which subtracts
a linear approximation from $f$ before searching for $\theta$.
The rationale for this alternative is that RQMC might already
do well integrating a linear function and we would then want
to choose a direction $\theta$ that performs well for the nonlinear part of $f$.
In our examples, we found very little difference between the two methods
and so we proposed the simpler uncentered active subspace pre-integration.

\section*{Acknowledgments}
This work was supported by the U.S.
National Science Foundation under grant IIS-1837931.

\appendix
\bibliographystyle{plain}
\bibliography{preintegration.bbl}

\begin{thebibliography}{10}

\bibitem{acwo:broa:glas:1997}
P.~Acworth, M.~Broadie, and P.~Glasserman.
\newblock A comparison of some {Monte Carlo} techniques for option pricing.
\newblock In H.~Niederreiter, P.~Hellekalek, G.~Larcher, and P.~Zinterhof,
  editors, {\em {Monte Carlo} and quasi-{Monte Carlo} methods '96}, pages
  1--18. Springer, 1997.

\bibitem{chen1982inequality}
L.~H.~Y. Chen.
\newblock An inequality for the multivariate normal distribution.
\newblock {\em Journal of Multivariate Analysis}, 12(2):306--315, 1982.

\bibitem{cons:2015}
P.~G. Constantine.
\newblock {\em Active subspaces: Emerging ideas for dimension reduction in
  parameter studies}.
\newblock SIAM, Philadelphia, 2015.

\bibitem{constantine2014active}
P.~G. Constantine, E.~Dow, and Q.~Wang.
\newblock Active subspace methods in theory and practice: applications to
  kriging surfaces.
\newblock {\em SIAM Journal on Scientific Computing}, 36(4):A1500--A1524, 2014.

\bibitem{davrab}
P.~J. Davis and P.~Rabinowitz.
\newblock {\em Methods of Numerical Integration}.
\newblock Academic Press, San Diego, 2nd edition, 1984.

\bibitem{dick:kuo:sloa:2013}
J.~Dick, F.~Y. Kuo, and I.~H. Sloan.
\newblock High-dimensional integration: the {quasi-Monte Carlo} way.
\newblock {\em Acta Numerica}, 22:133--288, 2013.

\bibitem{dick:pill:2010}
J.~Dick and F.~Pillichshammer.
\newblock {\em Digital sequences, discrepancy and quasi-{Monte Carlo}
  integration}.
\newblock Cambridge University Press, Cambridge, 2010.

\bibitem{efro:stei:1981}
B.~Efron and C.~Stein.
\newblock The jackknife estimate of variance.
\newblock {\em Annals of Statistics}, 9(3):586--596, 1981.

\bibitem{gilb:kuo:sloa:2021:tr}
A.~D. Gilbert, F.~Y. Kuo, and I.~H. Sloan.
\newblock Preintegration is not smoothing when monotonicity fails.
\newblock Technical report, arXiv:2112.11621, 2021.

\bibitem{glasserman2004monte}
P.~Glasserman.
\newblock {\em Monte Carlo Methods in Financial Engineering}, volume~53.
\newblock Springer, 2004.

\bibitem{glas:hied:shah:1999}
P.~Glasserman, P.~Heidelberger, and P.~Shahabuddin.
\newblock Asymptotically optimal importance sampling and stratification for
  pricing path-dependent options.
\newblock {\em Mathematical finance}, 9(2):117--152, 1999.

\bibitem{griewank2018high}
A.~Griewank, F.~Y. Kuo, H.~Le{\"o}vey, and I.~H. Sloan.
\newblock High dimensional integration of kinks and jumps—smoothing by
  preintegration.
\newblock {\em Journal of Computational and Applied Mathematics}, 344:259--274,
  2018.

\bibitem{hamm:1956}
J.~M. Hammersley.
\newblock Conditional {Monte Carlo}.
\newblock {\em Journal of the ACM (JACM)}, 3(2):73--76, 1956.

\bibitem{he2019error}
Z.~He.
\newblock On the error rate of conditional quasi--{M}onte {C}arlo for
  discontinuous functions.
\newblock {\em SIAM Journal on Numerical Analysis}, 57(2):854--874, 2019.

\bibitem{hick:2014}
F.~J. Hickernell.
\newblock {Koksma-Hlawka} inequality.
\newblock {\em Wiley StatsRef: Statistics Reference Online}, 2014.

\bibitem{hoef:1948}
W.~Hoeffding.
\newblock A class of statistics with asymptotically normal distribution.
\newblock {\em Annals of Mathematical Statistics}, 19(3):293--325, 1948.

\bibitem{hoyt:owen:2020}
C.~R. Hoyt and A.~B. Owen.
\newblock Mean dimension of ridge functions.
\newblock {\em SIAM Journal on Numerical Analysis}, 58(2):1195--1216, 2020.

\bibitem{imai:tan:2004}
J.~Imai and K.~S. Tan.
\newblock Minimizing effective dimension using linear transformation.
\newblock In {\em Monte Carlo and Quasi-Monte Carlo Methods 2002}, pages
  275--292. Springer, 2004.

\bibitem{jans:1999}
M.~J.~W. Jansen.
\newblock Analysis of variance designs for model output.
\newblock {\em Computer Physics Communications}, 117(1--2):35--43, 1999.

\bibitem{joll:1982}
I.~T. Jolliffe.
\newblock A note on the use of principal components in regression.
\newblock {\em Journal of the Royal Statistical Society: Series C (Applied
  Statistics)}, 31(3):300--303, 1982.

\bibitem{lecu:lemi:2002}
P.~L'Ecuyer and C.~Lemieux.
\newblock A survey of randomized quasi-{M}onte {C}arlo methods.
\newblock In {\em Modeling Uncertainty: An Examination of Stochastic Theory,
  Methods, and Applications}, pages 419--474. Kluwer Academic Publishers, 2002.

\bibitem{mato:1998:2}
J.~Matou\v{s}ek.
\newblock On the {L$^2$}--discrepancy for anchored boxes.
\newblock {\em Journal of Complexity}, 14(4):527--556, 1998.

\bibitem{mosk:cafl:1996}
B.~Moskowitz and R.~E. Caflisch.
\newblock Smoothness and dimension reduction in {quasi-Monte Carlo} methods.
\newblock {\em Mathematical and Computer Modelling}, 23(8-9):37--54, 1996.

\bibitem{nied:1992}
H.~Niederreiter.
\newblock {\em Random Number Generation and Quasi-{Monte Carlo} Methods}.
\newblock SIAM, Philadelphia, PA, 1992.

\bibitem{okte:gonc:2011}
G.~{\"O}kten and A.~G{\"o}nc{\"u}.
\newblock Generating low-discrepancy sequences from the normal distribution:
  {Box--Muller} or inverse transform?
\newblock {\em Mathematical and Computer Modelling}, 53(5-6):1268--1281, 2011.

\bibitem{rtms}
A.~B. Owen.
\newblock Randomly permuted $(t,m,s)$-nets and $(t,s)$-sequences.
\newblock In {\em Monte Carlo and Quasi-Monte Carlo Methods in Scientific
  Computing}, pages 299--317, New York, 1995. Springer-Verlag.

\bibitem{snetvar}
A.~B. Owen.
\newblock {Monte Carlo} variance of scrambled net quadrature.
\newblock {\em SIAM Journal of Numerical Analysis}, 34(5):1884--1910, 1997.

\bibitem{smoovar}
A.~B. Owen.
\newblock Scrambled net variance for integrals of smooth functions.
\newblock {\em Annals of Statistics}, 25(4):1541--1562, 1997.

\bibitem{snxs}
A.~B. Owen.
\newblock Scrambling {S}obol' and {N}iederreiter-{X}ing points.
\newblock {\em Journal of Complexity}, 14(4):466--489, December 1998.

\bibitem{localanti}
A.~B. Owen.
\newblock Local antithetic sampling with scrambled nets.
\newblock {\em Annals of Statistics}, 36(5):2319--2343, 2008.

\bibitem{mcbook}
A.~B. Owen.
\newblock {Monte Carlo} {T}heory, {M}ethods and {E}xamples.
\newblock \url{statweb.stanford.edu/~owen/mc}, 2013.

\bibitem{sllnrqmc}
A.~B. Owen and D.~Rudolf.
\newblock A strong law of large numbers for scrambled net integration.
\newblock {\em SIAM Review}, 63(2):360--372, 2021.

\bibitem{pan:owen:2021:tr}
Z.~Pan and A.~B. Owen.
\newblock The nonzero gain coefficients of {Sobol's} sequences are always
  powers of two.
\newblock Technical Report arXiv:2106.10534, Stanford University, 2021.

\bibitem{papa:2002}
A.~Papageorgiou.
\newblock The {Brownian} bridge does not offer a consistent advantage in
  {quasi-Monte Carlo} integration.
\newblock {\em Journal of Complexity}, 18(1):171--186, 2002.

\bibitem{parente2020generalized}
M.~T. Parente, J.~Wallin, and B.~Wohlmuth.
\newblock Generalized bounds for active subspaces.
\newblock {\em Electronic Journal of Statistics}, 14(1):917--943, 2020.

\bibitem{pirs:1995}
G.~Pirsic.
\newblock Schnell konvergierende {Walshreihen} {\"u}ber gruppen.
\newblock Master's thesis, University of Salzburg, 1995.
\newblock Institute for Mathematics.

\bibitem{raza:etal:2021}
S.~Razavi, A.~Jakeman, A.~Saltelli, C.~Prieur, B.~Iooss, E.~Borgonovo,
  E.~Plischke, S.~L. Piano, T.~Iwanaga, W.~Becker, S.~Tarantola, J.~H.~A.
  Guillaume, J.~Jakeman, H.~Gupta, N.~Melillo, G.~Rabitti, V.~Chabirdon,
  Q.~Duan, X.~Sun, S.~Smith, R.~Sheikholeslami, N.~Hosseini, M.~Asadzade,
  A.~Puy, S.~Kucherenko, and H.~R. Maier.
\newblock The future of sensitivity analysis: an essential discipline for
  systems modeling and policy support.
\newblock {\em Environmental Modelling \& Software}, 137:104954, 2021.

\bibitem{robe:robe:2021}
C.~P. Robert and G.~O. Roberts.
\newblock {Rao-Blackwellization} in the {MCMC} era.
\newblock Technical Report arXiv:2101.01011, University of Warwick, 2021.

\bibitem{sobo:1967:tran}
I.~M. Sobol'.
\newblock The distribution of points in a cube and the accurate evaluation of
  integrals.
\newblock {\em USSR Computational Mathematics and Mathematical Physics},
  7(4):86--112, 1967.

\bibitem{sobo:1969}
I.~M. Sobol'.
\newblock {\em Multidimensional Quadrature Formulas and {H}aar Functions}.
\newblock Nauka, Moscow, 1969.
\newblock (In Russian).

\bibitem{sobo:1993}
I.~M. Sobol'.
\newblock Sensitivity estimates for nonlinear mathematical models.
\newblock {\em Mathematical Modeling and Computational Experiment}, 1:407--414,
  1993.

\bibitem{im2009derivative}
I.~M. Sobol’ and S.~Kucherenko.
\newblock Derivative based global sensitivity measures and their link with
  global sensitivity indices.
\newblock {\em Mathematics and Computers in Simulation (MATCOM)},
  79(10):3009--3017, 2009.

\bibitem{trot:tuke:1956}
H.~F. Trotter and J.~W. Tukey.
\newblock Conditional {Monte Carlo} for normal samples.
\newblock In {\em Symposium on {Monte Carlo} Methods}, pages 64--79, New York,
  1956. Wiley.

\bibitem{xiao2018conditional}
Y.~Xiao and X.~Wang.
\newblock Conditional quasi-{M}onte {C}arlo methods and dimension reduction for
  option pricing and hedging with discontinuous functions.
\newblock {\em Journal of Computational and Applied Mathematics}, 343:289--308,
  2018.

\bibitem{xiao:wang:2019}
Y.~Xiao and X.~Wang.
\newblock Enhancing {quasi-Monte Carlo} simulation by minimizing effective
  dimension for derivative pricing.
\newblock {\em Computational Economics}, 54(1):343--366, 2019.

\bibitem{yue:mao:1999}
R.-X. Yue and S.-S. Mao.
\newblock On the variance of quadrature over scrambled nets and sequences.
\newblock {\em Statistics \& Probability Letters}, 44(3):267--280, 1999.

\bibitem{zhan:wang:he:2021}
C.~Zhang, X.~Wang, and Z.~He.
\newblock Efficient importance sampling in quasi-{Monte Carlo} methods for
  computational finance.
\newblock {\em SIAM Journal on Scientific Computing}, 43(1):B1--B29, 2021.

\end{thebibliography}

\end{document}